\documentclass[11pt,reqno]{amsart}
\makeatletter
\@ifundefined{subjclassname@2020}{%
  \@namedef{subjclassname@2020}{\textup{2020} Mathematics Subject Classification}}{}
\makeatother
\usepackage[T1]{fontenc}
\usepackage{lmodern}
\usepackage{amsmath,amssymb,mathtools,mathrsfs}
\usepackage[letterpaper,textwidth=6.35in,textheight=9.25in,centering]{geometry}
\usepackage{microtype}
\usepackage{enumitem}
\usepackage{xcolor}
\usepackage{needspace}
\usepackage[hidelinks]{hyperref}
\hypersetup{pdftitle={Sharp Moser-Onofri inequality on the CR sphere under zero-moment constraints and   stability of CR Moser-Onofri inequality},pdfauthor={Lu Chen and Guozhen Lu},pdfsubject={Global energy-distance stability from Theorem 1.1 by conformal centering, following Carlen, without a mass normalization},pdfkeywords={Moser-Onofri inequality, CR sphere, zero-moment constraints, extremal functions, Hardy space, sharp stability}}
\numberwithin{equation}{section}
\newtheorem{theorem}{Theorem}[section]
\newtheorem{lemma}[theorem]{Lemma}
\newtheorem{proposition}[theorem]{Proposition}
\newtheorem{corollary}[theorem]{Corollary}
\theoremstyle{remark}
\newtheorem{remark}[theorem]{Remark}
\newtheorem*{remark*}{Remark}
\theoremstyle{plain}
\newtheorem*{theoremA}{Theorem A}
\newtheorem*{theoremB}{Theorem B}
\newtheorem*{theoremC}{Theorem C}
\newcommand{\R}{\mathbb R}
\newcommand{\C}{\mathbb C}
\newcommand{\N}{\mathbb N}
\newcommand{\Sph}{\mathbb S^{2n+1}}
\newcommand{\Ball}{\mathbb B^{n+1}}
\newcommand{\PR}{\mathcal P_{\R}}
\newcommand{\Dn}{\mathcal D_n}
\newcommand{\A}{A'_Q}
\newcommand{\Ls}{\mathcal L}
\newcommand{\PP}{\pi_{\mathcal P}}
\newcommand{\E}{\mathcal E}
\newcommand{\Jf}{\mathcal J_{\alpha,n}}
\newcommand{\Jone}{\mathcal J_{1,n}}
\newcommand{\Jhalf}{\mathcal J_{1/2,n}}
\newcommand{\Lzero}{\mathcal L_0}
\newcommand{\mean}[1]{\langle #1\rangle_{\sigma}}
\newcommand{\ip}[2]{\langle #1,#2\rangle}
\newcommand{\norm}[1]{\lVert #1\rVert}
\newcommand{\abs}[1]{\lvert #1\rvert}
\newcommand{\dd}{\,d\sigma}
\newcommand{\cstr}{\mathbb J}
\newcommand{\Rcal}{\mathcal R}
\newcommand{\Scal}{\mathcal S}
\newcommand{\Tcal}{\mathcal T}

\newcommand{\Hn}{\mathbb H^n}

\DeclareMathOperator{\Rea}{Re}
\DeclareMathOperator{\Ima}{Im}
\DeclareMathOperator{\Div}{div}
\DeclareMathOperator{\supp}{supp}

\DeclareMathOperator{\Aut}{Aut}
\setlist[enumerate]{leftmargin=*,label=\textup{(\arabic*)},itemsep=3pt}
\allowdisplaybreaks[2]
\title[CR Moser--Onofri under zero-moment constraints]{Sharp Moser-Onofri inequality on the CR sphere under zero-moment constraints and its application to the stability of the CR Moser-Onofri inequality}
\author{Lu Chen}
\address{Key Laboratory of Algebraic Lie Theory and Analysis of Ministry of Education, School of Mathematics and Statistics, Beijing Institute of Technology, Haidian, 100081, Beijing, China}
\email{chenlu5818804@163.com}
\author{Guozhen Lu}
\address{Department of Mathematics, University of Connecticut, Storrs, Connecticut 06269, USA}
\email{guozhen.lu@uconn.edu}
\date{}
\begin{document}
\begin{abstract}
In this paper, we establish the sharp Moser--Onofri inequality in CR sphere under zero-moment constraints. By combining a decomposition into low and high spherical harmonics with the known CR localization estimates, we first prove the existence of extremals of the subcritical Moser--Onofri inequality under zero-moment constraints. We then derive their Euler--Lagrange equation, establish their smoothness, and eliminate the moment multiplier by means of the infinitesimal CR conformal action. Using the summation identity \eqref{eq:shift-sum} on the Hardy space $H^2(\Ball)$, we show that the extremals of subcritical Moser-Onofri inequality in CR sphere under zero-moment constraints are constants.
By a limiting argument, we deduce the sharp Moser--Onofri inequality in CR sphere under zero-moment constraints. As an application, we establish the stability of Moser--Onofri inequality in CR sphere with the best stability constant.

\end{abstract}
\maketitle
\begingroup
\raggedbottom
\tableofcontents
\clearpage
\endgroup

\section{Introduction and main results}
\subsection{The sharp Sobolev, Hardy--Littlewood--Sobolev, and Moser-Onofri inequalities}
The main purpose of this paper is to establish the sharp Moser-Onofri inequality on the standard CR sphere under the zero-moment condition for the exponential density and to derive a global stability inequality on the energy space that measures the deviation from the family of extremals. We begin by recalling the sharp Sobolev and Hardy-Littlewood-Sobolev (HLS) inequalities and the corresponding Moser-Onofri inequality. On ordinary spheres, the logarithmic endpoint inequalities include Onofri's inequality in dimension two and Beckner's higher-dimensional inequality \cite{Onofri,Beckner}. On the CR sphere, the corresponding inequalities are formulated in terms of the horizontal differential structure and the homogeneous dimension. We recall their precise forms and normalizations below.

Throughout the paper, $n\geq1$. The Heisenberg group associated with $\Sph\subset\C^{n+1}$ is
\[
\mathbb H^{n}=\C^{n}\times\R,
\qquad
(z,t)(z',t')=(z+z',t+t'+2\Ima(z\cdot\overline{z'})).
\]
Its homogeneous dimension is $Q=2n+2$. The dilations and homogeneous norm are
\[
\delta_r(z,t)=(rz,r^2t),\qquad
\abs{(z,t)}=(\abs z^4+t^2)^{1/4}.
\]
Writing $z_j=x_j+iy_j$, we set
\[
X_j=\partial_{x_j}+2y_j\partial_t,\qquad
Y_j=\partial_{y_j}-2x_j\partial_t,\qquad
\abs{\nabla_H u}^2=\sum_{j=1}^{n}(\abs{X_ju}^2+\abs{Y_ju}^2).
\]
We denote by $S^1(\mathbb H^n)$ the completion of $C_c^\infty(\mathbb H^n)$ under the norm $\norm{\nabla_H f}_2$.
Folland and Stein developed the Sobolev space theory on the Heisenberg group \cite{FS}. Jerison and Lee determined the sharp constant and classified the extremals \cite{JL87,JL88}. In the above normalization, their inequality takes the following form. It agrees with the Sobolev inequality in \cite{FL} after multiplying both sides by four.

\begin{theoremA}
For every $f\in S^1(\mathbb H^n)$, the Sobolev inequality on the Heisenberg group states
\[
\int_{\mathbb H^n}\abs{\nabla_H f}^2\,dz\,dt
\geq \frac{4\pi n^2}{(2^{2n}n!)^{1/(n+1)}}
\norm f_{2Q/(Q-2)}^2.
\]
The constant $4\pi n^2/(2^{2n}n!)^{1/(n+1)}$ is sharp.
Equality holds if and only if
\[
f(z,t)=K\abs{w+z\cdot\mu+\lambda}^{-n},
\]
where $w=t+i\abs z^2$, $K,\lambda\in\C$, $\mu\in\C^n$, and
$\Ima\lambda>\abs\mu^2/4$.
\end{theoremA}

Frank and Lieb established a family of sharp fractional inequalities \cite{FL}, including the following Hardy-Littlewood-Sobolev inequality.

\begin{theoremB}
Let $0<\lambda<Q$. For every $f\in L^{2Q/(2Q-\lambda)}(\mathbb H^n)$, the following inequality holds:
\[
\left|\iint_{\mathbb H^n\times\mathbb H^n}
\frac{\overline{f(u)}f(v)}{\abs{u^{-1}v}^{\lambda}}\,du\,dv\right|
\leq
\left(\frac{\pi^{n+1}}{2^{n-1}n!}\right)^{\lambda/Q}
\frac{n!\,\Gamma((Q-\lambda)/2)}{\Gamma((2Q-\lambda)/4)^2}
\norm f_{2Q/(2Q-\lambda)}^2.
\]
Equality holds if and only if
\[
f(u)=cF(\delta_r(a^{-1}u))
\]
for some $c\in\C$, $r>0$, and $a\in\mathbb H^n$, where
\[
F(u)=((1+\abs z^2)^2+t^2)^{-(2Q-\lambda)/4},\qquad u=(z,t).
\]
\end{theoremB}
\Needspace{20\baselineskip}
The Heisenberg group is conformally equivalent to the standard complex sphere, or CR sphere,
\[
 \Sph=\left\{\xi=(\xi_1,\ldots,\xi_{n+1})\in\C^{n+1}:|\xi|=1\right\}.
\]
The Cayley transform $\mathcal C:\Hn\to\Sph\setminus\{(0,\ldots,0,-1)\}$ is
\[
 \mathcal C(z,t)
 =\left(
 \frac{2z}{1+|z|^2+it},
 \frac{1-|z|^2-it}{1+|z|^2+it}
 \right),
\]
with inverse
\[
 \mathcal C^{-1}(\xi)
 =\left(
 \frac{\xi'}{1+\xi_{n+1}},
 \operatorname{Im}\frac{1-\xi_{n+1}}{1+\xi_{n+1}}
 \right),\qquad \xi'=(\xi_1,\ldots,\xi_n).
\]
With respect to the  Lebesgue measure on $\Hn$ and standard surface measure $d\xi$ on $\Sph$, its Jacobian is
\[
 |J_{\mathcal C}(z,t)|
 =2^{Q-1}\bigl((1+|z|^2)^2+t^2\bigr)^{-Q/2}
 =\frac12|1+\xi_{n+1}|^Q,
 \qquad \xi=\mathcal C(z,t).
\]
These formulas and the conformal transformation of the sharp inequalities can be found in \cite{JL88,FL,ChenLuTangWang2026}.

For $j=1,\ldots,n+1$, introduce the tangential fields
\begin{align*}
 T_j&=\frac{\partial}{\partial\xi_j}
 -\overline\xi_j\sum_{k=1}^{n+1}\xi_k\frac{\partial}{\partial\xi_k},
 &
 \overline T_j&=\frac{\partial}{\partial\overline\xi_j}
 -\xi_j\sum_{k=1}^{n+1}\overline\xi_k\frac{\partial}{\partial\overline\xi_k}.
\end{align*}
The order-two conformal sub-Laplacian and its quadratic form are
\begin{align*}
 \mathcal L_{\mathrm{conf}}
 &=-\frac12\sum_{j=1}^{n+1}
 (T_j\overline T_j+\overline T_jT_j)+\frac{n^2}{4},\\
 \mathcal E_2[u]&=\int_{\Sph}\overline u\,\mathcal L_{\mathrm{conf}}u\,d\xi,
\end{align*}
where $S^1(\Sph)$ is the  domain of the quadratic form $\mathcal E_2$; see \cite{Folland,JL88,FL}. We distinguish this operator from the nonnegative sublaplacian $\Ls$ used below:
\[
 \Ls=-\frac12\sum_{j=1}^{n+1}(T_j\overline T_j+\overline T_jT_j),
 \qquad \mathcal L_{\mathrm{conf}}=\Ls+\frac{n^2}{4}.
\]
The Cayley transform identifies the inequality in Theorem~A with the corresponding sharp Sobolev inequality on the CR sphere and identifies the inequality in Theorem~B with the sharp HLS family on the CR sphere. The bispherical harmonic decomposition diagonalizes the associated CR intertwining operators \cite{Folland,FL}.

For the exponential endpoint of this Sobolev--HLS theory on the CR sphere, Cohn and Lu \cite{CL} proved sharp Moser--Trudinger inequalities on complex spheres. Their proof uses a sharp representation formula involving complex tangential derivatives, and they also derived Moser--Onofri type inequalities as applications. Branson, Fontana, and Morpurgo identified the conditional intertwinor and established the conformally invariant Beckner-Onofri inequality on the space of real CR-pluriharmonic functions \cite{BFM}.

From now on we use the
probability surface measure $d\sigma=d\xi/|\Sph|$. With this measure the Jacobian of the Cayley transform is $|\Sph|^{-1}|J_{\mathcal C}|$. For functions in the energy space introduced below, we write
\[
\mean F=\int_{\Sph}F\dd,\qquad
\E(F)=\ip F{\A F},\qquad Z(F)=\int_{\Sph}e^F\dd.
\]
Thus $Z(F)$ denotes the integral of $e^F$ with respect to the normalized surface measure. Here $\A$ is the conditional intertwinor on $\PR$. Its eigenvalue on the pure holomorphic and antiholomorphic modes of degree $k\geq1$ is
$\lambda_k=\Gamma(k+n+1)/\Gamma(k)$, and it maps constant functions to zero. The space $\PR$ and the form domain $\Dn$ are defined in Section~\ref{sec:spectrum}. In this normalization, the sharp Moser--Onofri inequality takes the following form \cite{BFM}.

\Needspace{20\baselineskip}
\begin{theoremC}
For every real function $F\in\Dn$,
\[
\frac{1}{2(n+1)!}\E(F)+\mean F-\log Z(F)\geq0.
\]
The coefficient $1/(2(n+1)!)$ is sharp. The inequality is invariant under the conformal action
\[
F\longmapsto F^\tau=F\circ\tau+\log J_\tau,
\qquad \tau\in\Aut(\Sph).
\]
Equality holds if and only if
\[
F(\zeta)=c+\log J_\tau(\zeta)
\quad\text{for some }c\in\R\text{ and }\tau\in\Aut(\Sph).
\]
In particular, under the normalization $Z(F)=1$, equality holds if and only if $F=\log J_\tau$ for some $\tau\in\Aut(\Sph)$. The group $\Aut(\Sph)$ and its positive boundary Jacobian $J_\tau$ are defined in Subsection~\ref{sec:automorphisms}.
\end{theoremC}

The noncompact family of equality cases in Theorem~C motivates the study of sharp Moser-Onofri inequality under zero-moment condition
\begin{equation}\label{eq:center-intro}
\int_{\Sph}\zeta e^F\dd=0\quad\hbox{in }\C^{n+1}.
\end{equation}
We call a function satisfying this condition centered. Such moment conditions on the conformal density occur in the work of Aubin and Chang--Yang \cite{Aubin,CY87,CY95}. On Euclidean spheres, this improved inequality was called Chang-Yang conjecture. It was solved by Gui and Moradifam \cite{GM} for $n=2$ and Gui, Li, Wei, and Ye \cite{GLWY} for $n>2$. The latter uses an integral representation and the rigidity of stable critical points.

The main purpose of this paper is to solve the best constant problem for Moser-Onofri inequality under zero-moment condition on the CR sphere.
Define
\[
\Jf(F)=\frac{\alpha}{2(n+1)!}\E(F)+\mean F-\log Z(F),\qquad \alpha>0,
\]
and let $\Lzero$ be the set of functions in $\Dn$ satisfying \eqref{eq:center-intro}. Both $\Jf$ and the constraint are invariant under the addition of a constant. We derive that

\begin{theorem}\label{thm:main}
Let $n\geq1$. For every $F\in\Lzero$,
\[
\Jhalf(F)=\frac{1}{4(n+1)!}\E(F)+\mean F-\log Z(F)\geq0.
\]
The coefficient $1/(4(n+1)!)$ is sharp. Equality holds if and only if $F$ is constant.
\end{theorem}

To prove Theorem~\ref{thm:main}, we first establish the existence of extremals of the subcritical Moser--Onofri inequality under zero-moment constraints for $\alpha>1/2$. We then compare an identity from their Euler--Lagrange equation with the estimate in Theorem~\ref{thm:moment-estimate} and show that these extremals are constant for $1/2<\alpha\leq1$, which gives $\Jf(F)\geq0$ for every $F\in\Lzero$ for $1/2<\alpha\leq1$. Passing to the limit $\alpha\rightarrow \frac{1}{2}$, we obtain the sharp Moser--Onofri inequality under zero-moment constraints stated in Theorem~\ref{thm:main}.

In the process of proving that the extremals of the subcritical Moser--Onofri inequality under zero-moment constraints are constant, we in fact prove that any critical point of the functional $\Jf$ in $\Dn$ under zero-moment constraint is a constant for $1/2\leq\alpha\leq1$. We state this conclusion in the following corollary.
\begin{corollary}\label{cor:critical-rigidity}
Let $n\geq1$ and $1/2\leq\alpha\leq1$. Any critical point of $\Jf$ under the zero-moment constraint is a constant.
\end{corollary}

The linearization of the moment constraint at a constant removes the first spherical harmonics from the tangent space. After fixing the additive constant, the first remaining pure spherical harmonics have degree two. Consequently, the constrained Hessian is nonnegative exactly when $\alpha\geq1/(n+2)$; see Remark~\ref{rem:hessian}. However, the zero-moment condition does not imply that the degree-one component of an admissible $F$ vanishes.

As an application of Theorem~\ref{thm:main}, we also deduce the following sharp stability inequality for the Moser--Onofri inequality on the CR sphere.

\begin{theorem}\label{thm:stability}
Let $n\geq1$, and let
\[
\mathcal M=\{c+\log J_\tau:c\in\R,\ \tau\in\Aut(\Sph)\}
\]
be the set of extremals of sharp Moser-Onofri inequality in Theorem~C. For every real function $F\in\Dn$, there holds
\begin{equation}\label{eq:stability}
\frac{1}{2(n+1)!}\E(F)+\mean F-\log Z(F)
\geq\frac{1}{4(n+1)!}\inf_{h\in\mathcal M}\E(F-h).
\end{equation}
Equality holds if and only if $F\in\mathcal M$.
\end{theorem}

Theorem~\ref{thm:stability} applies to all functions on the whole energy space $\Dn$, without a zero-moment or exponential-mass constraint on $F$. The distance from $F$ to extremal manifold
$\mathcal M$ is measured by the quadratic form of the conformally
invariant operator $\A$. We also note that
in classical Euclidean spheres, Chen, Lu, and Tang \cite{ChenLuTang} proved the local stability
for the log-Sobolev and Moser-Onofri inequalities.
The nonexistence result for global Moser--Onofri stability
concerns the $L^2$ distance of $u=e^{F/2}$ to the conformal
extremal family. In stability inequality \eqref{eq:stability}, the distance is measured in the energy
norm of $F$ to the extremal family. Thus the two results concern different distances and
different classes of admissible functions.

\subsection{Main steps of the proof}
We first establish the existence of minimizers for $\alpha>1/2$.
Branson, Fontana, and Morpurgo proved an Aubin-type inequality under the
zero-moment condition in \cite{BFM}. Their estimate
contains the additional lower-order term
$C_1(\beta,n)\norm{\Ls^{(Q-1)/4}F}_2^2$; see
\eqref{eq:BFM-Aubin-remainder}. Since this term depends on $F$, it
cannot simply be included in an additive constant independent of $F$.
For the direct-method proof used here, we remove this lower-order term.
More precisely, Lemma~\ref{lem:subcritical} proves that, for every
$\beta>1/2$, there is a constant $C_{\beta,n}\geq0$, independent of $F$,
such that
\[
\log Z(F)-\mean F
\leq\frac{\beta}{2(n+1)!}\E(F)+C_{\beta,n},
\qquad F\in\Lzero.
\]
The additive constant is allowed at this stage. The sharp lower bound
zero for $\mathcal J_{\beta,n}$ is obtained after the existence and
classification of extremals have been proved.

To remove the lower-order term, we use the full-space exponential estimate
and the cutoff estimates in \cite{BFM}.
After subtracting $\mean F$, we write $F=U+V$, where $U$ is the
orthogonal projection onto the real pluriharmonic modes of degrees
$1\leq k\leq K$. Here $k$ is the degree index, while the integer
$K$ will be chosen independently of $F$. The zero-moment condition
for the full density $e^F/Z(F)$
provides two separated spherical caps, each carrying a fixed positive
amount of mass. On the set $A_i$ where the corresponding cutoff
$\chi_i$ equals one,
\[
b_iZ(F)\leq\int_{A_i}e^{U+V}\dd
\leq e^{\norm U_\infty}\int_{\Sph}e^{\chi_iV}\dd.
\]
We therefore apply the exponential and cutoff estimates to $V$ without
assuming that $e^V$ satisfies the zero-moment condition. On the modes of
degree larger than $K$, the lower-order remainder is bounded by a
quantity tending to zero as $K\to\infty$  times $\E(V)$. For fixed $K$,
$\norm U_\infty\leq B_K\E(U)^{1/2}$. Absorbing the high-degree
remainder and completing the square in $\E(U)^{1/2}$ gives the displayed
subcritical estimate with no lower-order term.

For a fixed $\alpha>1/2$, choose $1/2<\beta<\alpha$. The difference
between the two parameters then gives
\[
\Jf(F)=\mathcal J_{\beta,n}(F)
+\frac{\alpha-\beta}{2(n+1)!}\E(F)
\geq\frac{\alpha-\beta}{2(n+1)!}\E(F)-C_{\beta,n}.
\]
This gives a uniform energy bound for every minimizing sequence. Compactness and
uniform integrability of the exponential densities then give a minimizer
in Proposition~\ref{prop:existence}.
\vskip0.1cm

We next derive the Euler--Lagrange equation satisfied by these extremals.
We establish differentiability in the energy space and apply
the implicit function theorem to construct local correction maps.
These maps provide curves that preserve the mass and zero-moment
constraints, as well as smooth approximations
satisfying the same constraints.
The Euler--Lagrange equation initially contains a vector
Lagrange multiplier.
Using the spectral decomposition of $\A$, we obtain a gain
of $n+1$ derivatives in the ordinary Sobolev spaces on $S^{2n+1}$.
Since $n+1>(2n+1)/2$, Sobolev composition estimates and iteration
show that the extremals are smooth.
We then test the Euler--Lagrange equation with variations
induced by CR automorphisms.
Differentiating the moment constraint along these variations
gives a matrix identity which, together with the degree-one
equation, forces the multiplier to vanish for $0<\alpha\leq1$.
Since these variations need not preserve the moment constraint,
the multiplier term must be retained in this calculation.
\vskip0.1cm

To classify the extremals, we normalize $Z(F)=1$ and write
$F=f+\bar f$ and $h=e^f$.
The Euler--Lagrange equation gives an identity relating the
holomorphic moments of $\abs h^2$ to the radial energy of $h$.
Using the summation identity \eqref{eq:shift-sum} on the Hardy space $H^2(\Ball)$,
together with the zero-moment condition and the Cauchy--Schwarz inequality, we obtain
\[
\sum_{\abs\beta\geq2}\frac{(\abs\beta)!}{\beta!}
\left|\int_{\Sph}\zeta^\beta\abs h^2\dd\right|^2
\leq\frac1{2(n+1)}\sum_{m\geq1}m\norm{h_m}_2^2,
\]
where $\beta$ is a multi-index and $h_m$ is the homogeneous component
of $h$ of degree $m$; see \eqref{eq:moment-estimate}.
The equality case follows from the linear independence
of $z_1h$ and $z_2h$.
Combining the moment identity with this inequality shows that every
extremal of the subcritical Moser--Onofri inequality under zero-moment constraints
is constant for $1/2<\alpha\leq1$.
Since the functional vanishes at constants, its minimum is zero.
The equality case of the moment estimate also shows that any minimizer
at $\alpha=1/2$ must be constant.
\vskip 0.1cm

Finally, for each fixed admissible function, we obtain the
inequality at $\alpha=1/2$ by letting $\alpha\downarrow1/2$.
Any function attaining equality at the endpoint is a constrained
minimizer, so Proposition~\ref{prop:extremal-classification} gives its classification.
To prove sharpness, we use the family $F_r$ in \eqref{eq:two-pole}.
Since $F_r(-\zeta)=F_r(\zeta)$, its first moments vanish for every
$0<r<1$. The normalized densities $e^{F_r}/Z(F_r)$ concentrate at
two antipodal points as $r\uparrow1$, and the resulting asymptotics
show that the coefficient in Theorem~\ref{thm:main} cannot be decreased.
\medskip

To deduce Theorem~\ref{thm:stability}, we first choose a CR automorphism
$\tau$ such that $F_\tau=F\circ\tau+\log J_\tau$ has zero first moment.
Theorem~\ref{thm:main} gives
$\Jone(F_\tau)\geq\E(F_\tau)/(4(n+1)!)$.
The functional is invariant under this affine action, whereas the energy
is invariant under composition. The inverse-Jacobian identity therefore gives
$\E(F_\tau)=\E(F-\log J_{\tau^{-1}})$.
Since $\log J_{\tau^{-1}}\in\mathcal M$, taking the infimum over
$\mathcal M$ proves the asserted stability inequality.

\subsection{Organization}
This paper is organized as follows.
Section~2 introduces the CR automorphism group and its Jacobians,
the spectral and Hardy-space preliminaries, and the localization estimates.
In Section~3, we establish differentiability and smooth approximation
in the energy space, prove the subcritical inequality and the existence of minimizers,
and derive the Euler--Lagrange equation satisfied by the extremals.
In Section~4, we combine an identity from the Euler--Lagrange equation
with an estimate for the moments of $e^F$ to classify the extremals.
This proof also establishes Corollary~\ref{cor:critical-rigidity}.
Finally, Sections~5 and~6 give the proofs of the two main theorems respectively.

\section{Preliminaries}
In this section, we introduce the function spaces and operator normalizations, define the CR automorphism group and its boundary Jacobians, and record the conformal action. We then recall the exponential and cutoff estimates needed for the existence of extremals and the Hardy-space formulas needed for their classification.

\subsection{Spherical harmonics and the energy space}\label{sec:spectrum}
We use the complex $L^2$ inner product
\[
\ip uv=\int_{\Sph}\overline u\,v\dd,
\]
which is conjugate-linear in its first argument. For real functions it agrees with the real bilinear pairing. The standard orthogonal decomposition is
\begin{equation}\label{eq:bispherical}
L^2(\Sph;\C)=\widehat\bigoplus_{j,k\geq0}\mathcal H_{j,k},
\end{equation}
where $\mathcal H_{j,k}$ consists of the restrictions to the sphere of homogeneous harmonic polynomials of degree $j$ in $\zeta$ and degree $k$ in $\bar\zeta$. We use the decomposition, projection kernels, and operator normalizations of \cite{Folland,BFM,FL}. A real pluriharmonic function has the expansion
\begin{equation}\label{eq:PR}
F=\mean F+\sum_{k\geq1}(f_k+\bar f_k),
\qquad f_k\in\mathcal H_{k,0},
\end{equation}
with convergence in $L^2$. The set of such functions is denoted by $\PR$.

Let $\Pi_k$ project onto $\mathcal H_{k,0}$, and let $\Pi$ be the Hardy--Szeg\H{o} projection, including constants. For real $G\in L^2$, the real orthogonal projection onto $\PR$ is
\[
\PP G=\Pi G+\overline{\Pi G}-\mean G.
\]
The spectrum of the conditional intertwinor is
\begin{equation}\label{eq:spectrum}
\A Y_{k,0}=\lambda_kY_{k,0},\qquad
\A Y_{0,k}=\lambda_kY_{0,k},\qquad
\lambda_k=\frac{\Gamma(k+n+1)}{\Gamma(k)}=\prod_{j=0}^{n}(k+j),
\end{equation}
for $k\geq1$, and $\A1=0$; see \cite{BFM}. Thus
\begin{equation}\label{eq:form-domain}
\E(F)=2\sum_{k\geq1}\lambda_k\norm{f_k}_2^2,\qquad
\Dn=\{F\in\PR:\E(F)<\infty\},\qquad
\norm F_{\Dn}^2=\abs{\mean F}^2+\E(F).
\end{equation}
We call this real Hilbert space the energy space and the norm in \eqref{eq:form-domain} the energy norm. On the zero-mean subspace this norm is $\E(F)^{1/2}$. The notation $\ip{\A F}\phi$ denotes the pairing associated with the quadratic form when $\A F$ has only distributional meaning. Spectral truncations converge in $\Dn$, and $\lambda_1=(n+1)!$ gives
\[
\norm{F-\mean F}_2^2\leq\frac1{(n+1)!}\E(F).
\]

For the nonnegative CR sublaplacian $\Ls=\mathcal L_{\mathrm{conf}}-n^2/4$ introduced in Section~1, we recall its spectrum to relate the form domain to the standard Sobolev spaces. Its eigenvalue on $\mathcal H_{j,k}$ is
\[
 jk+\frac{n}{2}(j+k).
\]
In particular, on $\mathcal H_{k,0}\oplus\mathcal H_{0,k}$ it is
\begin{equation}\label{eq:L-eigenvalue}
\nu_k=\frac{n}{2}k.
\end{equation}
The Folland--Stein space of order $s$ and exponent two has spectral norm equivalent to $\norm{(I+\Ls)^{s/2}F}_2$. Therefore $\Dn$ is precisely the real pluriharmonic subspace of this Folland--Stein space for $s=Q/2=(n+1)$, with equivalent norms. This identifies the domain with the one in \cite{BFM}. We denote the ordinary $L^2$ Sobolev space on the sphere by $W^{s,2}(\Sph)$. On the pluriharmonic subspace, the energy norm is equivalent to the $W^{(n+1)/2,2}(\Sph)$ norm: the eigenvalue of $I-\Delta_{\Sph}$ on degree $k$ is $1+k(k+2n)$, whose $(n+1)/2$ power is comparable to $\lambda_k$. Since $(n+1)/2<(2n+1)/2$ for $n\geq1$, the Sobolev embedding at this order does not give continuity. In Section~3, the Euler--Lagrange equation will give $F\in W^{n+1,2}(\Sph)$, which allows us to establish smoothness.

\subsection{CR automorphisms and boundary Jacobians}\label{sec:automorphisms}
Let $\Ball=\{z\in\C^{n+1}:|z|<1\}$. We denote by $\Aut(\Sph)$ the group of CR automorphisms of $\Sph$. More precisely, a member $\tau$ is a smooth diffeomorphism of $\Sph$ that preserves the complex tangent spaces
\[
H_\zeta=\{v\in\C^{n+1}:v\cdot\bar\zeta=0\}
\]
and satisfies $d\tau_\zeta(iv)=i\,d\tau_\zeta(v)$ for $v\in H_\zeta$. Equivalently, its members are the boundary restrictions of biholomorphic automorphisms of $\Ball$, which extend smoothly to the closed ball; see \cite{Rudin} and \cite{BFM}. Thus $\Aut(\Sph)$ refers to the CR automorphism group, not the group of all smooth diffeomorphisms of the sphere.

For $\tau\in\Aut(\Sph)$, its positive boundary Jacobian $J_\tau$ is defined by
\begin{equation}\label{eq:boundary-Jacobian}
\tau^*(d\sigma)=J_\tau\,d\sigma,
\qquad
\int_{\Sph}\psi(\tau(\zeta))J_\tau(\zeta)\,d\sigma_\zeta
=\int_{\Sph}\psi(\eta)\,d\sigma_\eta
\end{equation}
for every integrable $\psi$. In particular,
\[
\int_{\Sph}J_\tau\dd=1,
\qquad
J_{\tau^{-1}}(\eta)=\frac{1}{J_\tau(\tau^{-1}(\eta))}.
\]
The same function $J_\tau$ is obtained with standard surface measure $d\xi$, since the source and target measures are rescaled by the same constant.

If $\tau$ is the boundary value of a ball automorphism $\Phi$ and $a=\Phi^{-1}(0)\in\Ball$, the standard Jacobian formula is
\begin{equation}\label{eq:automorphism-Jacobian}
J_\tau(\zeta)
=\left(\frac{1-|a|^2}{|1-\zeta\cdot\bar a|^2}\right)^{n+1}
=\frac{(1-|a|^2)^{n+1}}{|1-\zeta\cdot\bar a|^Q};
\end{equation}
see \cite{BFM}. Unitary automorphisms correspond to $a=0$ and have $J_\tau=1$. This formula also shows that $\log J_\tau$ is smooth and real CR-pluriharmonic: $|a|<1$ ensures that $1-z\cdot\bar a$ has a holomorphic logarithm on a neighborhood of the closed ball.

The equality family in Theorem~C is therefore
\begin{equation}\label{eq:Onofri-extremal-family}
\mathcal M=\{c+\log J_\tau:c\in\R,\ \tau\in\Aut(\Sph)\}.
\end{equation}
Recall that $Z(F)=\int_{\Sph}e^F\dd$ is the exponential integral introduced in Section~1. For $F=c+\log J_\tau$, we have
\[
Z(F)=e^c\int_{\Sph}J_\tau\dd=e^c.
\]
Hence $c=\log Z(F)$, and the equality functions with $Z(F)=1$ are precisely $\log J_\tau$.
The affine conformal action preserves $\PR$, $\Dn$, the exponential mass, and the unconstrained functional:
\begin{equation}\label{eq:affine-action}
F_\tau=F\circ\tau+\log J_\tau,\qquad
Z(F_\tau)=Z(F),\qquad \Jone(F_\tau)=\Jone(F).
\end{equation}
These are the covariance identities in \cite{BFM}. We differentiate them with respect to a CR automorphism only for smooth $F$.

We also record the corresponding identity for the quadratic form:
\begin{equation}\label{eq:energy-composition}
\E(V\circ\tau)=\E(V),\qquad V\in\Dn,\quad\tau\in\Aut(\Sph).
\end{equation}
It follows directly from \eqref{eq:affine-action}. Indeed, for every real $t$,
apply that identity to $tV$ and cancel the equal logarithmic masses to obtain
\[
\begin{split}
&\frac{1}{2(n+1)!}\E(tV\circ\tau+\log J_\tau)
+t\mean{V\circ\tau}+\mean{\log J_\tau}\\
&\hspace{35mm}=\frac{t^2}{2(n+1)!}\E(V)+t\mean V.
\end{split}
\]
Both sides are polynomials in $t$. Equality of their quadratic coefficients
is \eqref{eq:energy-composition}. This argument is valid on $\Dn$, since
\eqref{eq:affine-action} preserves that space and $Z(tV)<\infty$ for every
fixed $t$. For two functions transformed by the same affine action, it gives
\[
\E(F_\tau-h_\tau)=\E((F-h)\circ\tau)=\E(F-h).
\]
Thus the energy of a difference is preserved. The energy of a single
function under the affine action need not be preserved.

We shall use the standard CR center-of-mass lemma in its density form
\cite{FL}: for every $p\in L^1(\Sph)$ with $p\geq0$ and
$\int p\dd=1$, there is $\tau\in\Aut(\Sph)$ such that
\begin{equation}\label{eq:CR-centering-density}
\int_{\Sph}\zeta\,p(\tau(\zeta))J_\tau(\zeta)\,d\sigma_\zeta=0.
\end{equation}
This is the CR counterpart of the conformal centering used in
\cite{Carlen25}. For $F\in\Dn$, applying it to
$p=e^F/Z(F)$ gives $F_\tau\in\Lzero$ and $Z(F_\tau)=Z(F)$.
Normalizing this auxiliary density does not impose any condition on the
value of $Z(F)$.

\subsection{The known endpoint inequality}
In the preceding normalization, the theorem in \cite{BFM} is
\begin{equation}\label{eq:BFM}
\log Z(F)\leq\mean F+\frac1{2(n+1)!}\E(F),\qquad F\in\Dn.
\end{equation}
It includes the equality statement recalled in Theorem~C. Applying it to $tF$ and $-tF$, for $t>0$, gives
\begin{equation}\label{eq:exponential-bound}
\int_{\Sph}e^{t\abs F}\dd
\leq2\exp\left(t\abs{\mean F}+\frac{t^2}{2(n+1)!}\E(F)\right).
\end{equation}
In particular $Z(F)$ is finite and strictly positive. We write
\[
\Lzero=\left\{F\in\Dn:\int_{\Sph}xe^F\dd=0\right\},
\qquad x\in\R^Q\ \text{corresponding to }\zeta\in\C^{n+1}.
\]
We note that the real and complex moment conditions are equivalent.

\subsection{Estimates for the localization}\label{sec:localization}
The Aubin-type inequality under the zero-moment condition proved in
\cite{BFM} contains a lower-order term. In the present
notation, for every $\beta>1/2$ there are nonnegative constants
$C_1(\beta,n)$ and $C_2(\beta,n)$ such that
\begin{equation}\label{eq:BFM-Aubin-remainder}
\mathcal J_{\beta,n}(F)
+C_1(\beta,n)\norm{\Ls^{(Q-1)/4}F}_2^2
+C_2(\beta,n)\geq0,
\qquad F\in\Lzero.
\end{equation}
The constants are independent of $F$, but the lower-order term is not.
Thus \eqref{eq:BFM-Aubin-remainder} is not yet the bound
$\mathcal J_{\beta,n}(F)\geq-C_{\beta,n}$ needed below.
In Lemma~\ref{lem:subcritical}, we remove this additional term and obtain
\eqref{eq:subcritical-no-remainder}, with only an additive constant
independent of $F$. Its proof uses the full-space exponential estimate
and the cutoff estimates in \cite{BFM}.
For the reader's convenience, we record these two estimates below in
our notation; the dimension parameter of that paper agrees with the
present $n$.

Set $\kappa=2/n$ and, for $\lambda>0$, define
\[
\mathcal L_\lambda=\kappa\Ls\PP+\lambda^{1/(n+1)}\Ls(I-\PP),
\qquad E_\lambda(G)=\ip G{\mathcal L_\lambda^{n+1}G}.
\]
The projections commute with $\Ls$. By \cite{BFM}, every smooth real-valued $G$ satisfies
\begin{equation}\label{eq:aux-endpoint}
\log\int_{\Sph}e^G\dd\leq\mean G+a_n(\lambda)E_\lambda(G)+C_\lambda,
\end{equation}
where
\[
a_n(\lambda)=\frac1{2(n+1)!}\left[1+\frac1\lambda
\sum_{k=1}^\infty\frac{(k+n-1)!}{(n-1)!k![k+n/2]^{n+1}}\right]
\longrightarrow\frac1{2(n+1)!}\quad(\lambda\to\infty).
\]
The summand is of order $k^{-2}$. The constant $C_\lambda$ is independent of $G$. The estimate \eqref{eq:aux-endpoint} applies to smooth real-valued functions without requiring them to belong to $\PR$. This will allow us to apply it after multiplication by a cutoff function.

We next specialize the cutoff estimates \cite{BFM}. For a fixed real smooth cutoff function $\chi$, a real smooth $V\in\PR$ with zero mean, and $\epsilon>0$, they give
\[
E_\lambda(\chi V)\leq D_\chi(V)+\epsilon E_\lambda(V)
 +C_{n,\lambda,\epsilon,\chi}\ip V{\Ls^{n}V},
\]
where
\begin{equation}\label{eq:cutoff-leading}
D_\chi(V)=\kappa^{n+1}
\begin{cases}
\displaystyle\int_{\Sph}\chi^2\abs{\Ls^{(n+1)/2}V}^2\dd,& n\text{ odd},\\[4pt]
\displaystyle\int_{\Sph}\chi^2\abs{\nabla_H\Ls^{n/2}V}^2\dd,& n\text{ even}.
\end{cases}
\end{equation}
Here the sphere's horizontal gradient is normalized by $\int\abs{\nabla_H W}^2\dd=\ip W{\Ls W}$. In the quoted estimates the terms involving $(I-\PP)V$ vanish because $V\in\PR$. The remaining error has the form $\ip V{\Ls^{n}V}$ displayed in the formulas in \cite{BFM}. The cutoff product is treated in the full function space, since multiplication by $\chi$ need not preserve $\PR$.

For two cutoffs with $\chi_1^2+\chi_2^2\leq1$, the leading terms satisfy
\[
D_{\chi_1}(V)+D_{\chi_2}(V)\leq\kappa^{n+1}\ip V{\Ls^{n+1}V}=E_\lambda(V).
\]
This follows directly from \eqref{eq:cutoff-leading}, using the horizontal-gradient identity when $n+1$ is odd. Moreover, \eqref{eq:spectrum} and \eqref{eq:L-eigenvalue} imply
\[
E_\lambda(V)=\ip V{(\kappa\Ls)^{n+1}V}\leq\E(V),
\]
because $k^{n+1}\leq\lambda_k$. Consequently,
\begin{equation}\label{eq:two-cutoff-bound}
\sum_{i=1}^2E_\lambda(\chi_iV)
\leq(1+2\epsilon)\E(V)
 +C_{n,\lambda,\epsilon}\ip V{\Ls^{n}V}.
\end{equation}
The last constant is uniform when the cutoffs are chosen from a fixed finite family. We will use this estimate in the proof of Lemma~\ref{lem:subcritical}.

\subsection{Hardy-space formulas}
We use the complex Hardy space $H^2(\Ball)$, consisting of holomorphic functions $h$ on $\Ball$ such that
\[
\sup_{0<r<1}\int_{\Sph}\abs{h(r\zeta)}^2\dd<\infty.
\]
We identify these functions with their $L^2$ boundary values and use the boundary $L^2(d\sigma)$ norm and the same inner-product convention. For a multi-index $\gamma\in\N_0^{n+1}$, set $\gamma!=\prod_j\gamma_j!$ and $\abs\gamma=\sum_j\gamma_j$. The standard monomial norms and normalized monomials are
\begin{equation}\label{eq:monomial-norm}
\norm{z^\gamma}_2^2=\frac{n!\gamma!}{(n+\abs\gamma)!},\qquad
e_\gamma=\left(\frac{(n+\abs\gamma)!}{n!\gamma!}\right)^{1/2}z^\gamma.
\end{equation}
The functions $e_\gamma$ form an orthonormal basis of $H^2(\Ball)$. The dimension and reproducing kernel of $\mathcal H_{k,0}$ are
\begin{equation}\label{eq:reproducing-kernel}
d_k=\binom{k+n}{n},\qquad
K_k(\zeta,\eta)=d_k(\zeta\cdot\bar\eta)^k.
\end{equation}
These formulas, including the normalization by probability measure, are the usual Hardy-space formulas; see \cite{Rudin} and \cite{Zhu}.

Let $S_jh=z_jh$ and write $\varepsilon_j$ for the $j$th unit multi-index. The coordinate multiplication operators are bounded, and
\begin{equation}\label{eq:one-shift}
S_j^*e_\gamma=\left(\frac{\gamma_j}{n+\abs\gamma}\right)^{1/2}e_{\gamma-\varepsilon_j}
\quad\text{if }\gamma_j>0,
\end{equation}
with value zero if $\gamma_j=0$. This is the standard adjoint formula for coordinate multiplication; it also follows by taking the ratio of the two monomial norms in \eqref{eq:monomial-norm}. Let $h=\sum_{m\geq0}h_m$ be the homogeneous expansion of $h$, and set
\[
\Rcal(h)=\sum_{m\geq1}m\norm{h_m}_2^2.
\]
For smooth boundary values this equals $\ip h{Rh}$, where $R=\sum_jz_j\partial_{z_j}$ is the holomorphic radial derivative. The summation identity \eqref{eq:shift-sum} on the Hardy space $H^2(\Ball)$ will be proved in Section~4.

\section{The subcritical Moser-Onofri inequality under zero-moment constraint and existence of their extremals}
In this section, we establish the subcritical Moser--Onofri inequality under zero-moment constraints and prove the existence of its extremals, that is, the attainment of the infimum of $\Jf$ on $\Lzero$ for $\alpha>1/2$. We then derive their Euler--Lagrange equation, establish their smoothness, and eliminate the moment multiplier in the range $1/2<\alpha\leq1$ needed for their classification. For this purpose, we first prove the following lemma, which provides smooth approximation and local variations preserving the constraints. We write $M(F)=\int_{\Sph}xe^F\dd$.

\begin{lemma}\label{lem:correction}
For every $1\leq q<\infty$, the map $F\mapsto e^F$ from $\Dn$ to $L^q(\Sph)$ is twice continuously Fr\'echet differentiable, with
\begin{equation}\label{eq:exp-derivatives}
D(e^F)[H]=e^FH,\qquad D^2(e^F)[H,K]=e^FHK.
\end{equation}
For each $F\in\Lzero$, there are a neighborhood $\mathcal U_F$ of $F$ in $\Dn$ and $C^2$ maps $c:\mathcal U_F\to\R$ and $b:\mathcal U_F\to\R^Q$, with $c(F)=0$ and $b(F)=0$, such that
\begin{equation}\label{eq:local-correction}
\mathcal C_F(G)=G+c(G)+x\cdot b(G),\qquad
Z(\mathcal C_F(G))=Z(F),\quad M(\mathcal C_F(G))=0.
\end{equation}
In particular, there are smooth pluriharmonic functions $F_m\in\Lzero$ such that
\begin{equation}\label{eq:approx-conclusions}
Z(F_m)=Z(F),\qquad F_m\to F\text{ in }\Dn,
\qquad e^{F_m}\to e^F\text{ in every finite }L^p.
\end{equation}
\end{lemma}
\begin{proof}
The exponential estimate \eqref{eq:exponential-bound}, applied on the unit ball of $\Dn$ and followed by rescaling, gives
\begin{equation}\label{eq:form-Lp}
\norm H_p\leq C_{p,n}\norm H_{\Dn}\qquad(1\leq p<\infty).
\end{equation}
It also bounds $\norm{e^F}_p$ uniformly on bounded subsets of $\Dn$. Thus H\"older's inequality and $|e^F-e^G|\leq|F-G|(e^F+e^G)$ yield
\begin{equation}\label{eq:exp-continuity}
\norm{e^F-e^G}_q\leq C_{n,q,R}\norm{F-G}_{\Dn},
\qquad \norm F_{\Dn},\norm G_{\Dn}\leq R.
\end{equation}
For $\norm F_{\Dn}\leq R$ and $\norm K_{\Dn}\leq1$, Taylor's formula similarly gives
\begin{equation}\label{eq:exp-Taylor-error}
\norm{e^{F+K}-e^F-e^FK}_q
\leq\tfrac12\norm{e^F}_{3q}\norm{e^{|K|}}_{3q}\norm K_{6q}^2
\leq C_{n,q,R}\norm K_{\Dn}^2.
\end{equation}
This proves the first derivative in \eqref{eq:exp-derivatives}. Applying the same estimate in $L^{2q}$ and multiplying by $H$ proves differentiability of $F\mapsto D(e^F)$ in operator norm. The continuity of the second derivative follows from \eqref{eq:exp-continuity} and H\"older's inequality. In particular, $Z$, $M$, and $\Jf$ are $C^2$ on $\Dn$.

Fix $F\in\Lzero$, put $\rho_F=e^F/Z(F)$, and consider
\[
\Psi(G,c,b)=\bigl(Z(G+c+x\cdot b)-Z(F),\ M(G+c+x\cdot b)\bigr).
\]
This map is $C^2$, vanishes at $(F,0,0)$, and satisfies
\begin{equation}\label{eq:constraint-matrix}
D_{(c,b)}\Psi(F,0,0)=Z(F)
\begin{pmatrix}1&0\\0&C_F\end{pmatrix},
\qquad C_F=\int_{\Sph}xx^T\rho_F\dd.
\end{equation}
The matrix $C_F$ is positive definite: for $v\ne0$, the function $(x\cdot v)^2$ is positive outside a set of surface measure zero, and $\rho_F>0$ almost everywhere. The implicit-function theorem therefore gives \eqref{eq:local-correction}. Differentiating the constraints at $G=F$ yields
\[
Dc(F)[H]=-\int H\rho_F\dd,\qquad
Db(F)[H]=-C_F^{-1}\int xH\rho_F\dd.
\]
Consequently, if $\int H\rho_F\dd=0$ and $\int xH\rho_F\dd=0$, then $F_t=\mathcal C_F(F+tH)$ preserves both constraints for all sufficiently small $t$ and satisfies $F_0=F$ and $\dot F_0=H$.

Finally, let $T_mF$ be the spectral truncation of $F$ to degrees at most $m$, including the constant term. Since $T_mF\to F$ in $\Dn$, the functions $F_m=\mathcal C_F(T_mF)$ are well defined for large $m$, smooth, and pluriharmonic. The continuity of $c$ and $b$ gives $F_m\to F$ in $\Dn$, while \eqref{eq:local-correction} preserves the constraints and \eqref{eq:exp-continuity} gives the stated convergence of the exponentials.
\end{proof}

When $\mean F=0$, subtracting $\mean{F_m}$ from these approximations preserves the zero-moment condition and convergence in $\Dn$, although it may change the exponential mass. We will use either the zero-mean or the mass-one normalization as needed. The same lemma also gives the following second-variation formula at a constant.

\begin{remark}\label{rem:hessian}
At $F=0$, the tangent conditions are $\int H\dd=0$ and $\int xH\dd=0$. Lemma~\ref{lem:correction} supplies a $C^2$ admissible curve with velocity $H$. Since $D\Jf(0)=0$ on the full energy space, the second variation along this curve is
\begin{equation}\label{eq:hessian}
D^2\Jf(0)[H,H]=\frac\alpha{(n+1)!}\E(H)-\norm H_2^2.
\end{equation}
The tangent conditions remove the constant and degree-one components. Since $\lambda_2=(n+2)!$, this form is nonnegative on the tangent space exactly when $\alpha\geq1/(n+2)$; on a pure degree-two variation it equals $(\alpha(n+2)-1)\norm H_2^2$. At $\alpha=1/(n+2)$, nonnegativity of the Hessian alone does not imply nonlinear local minimality.
\end{remark}

We now prove the subcritical Moser--Onofri inequality under the zero-moment condition. Lemma~\ref{lem:subcritical} removes the lower-order term in \eqref{eq:BFM-Aubin-remainder} and
will help to establish the existence of corresponding extremals using the standard compact methods.

\begin{lemma}\label{lem:subcritical}
For every $\beta>1/2$, there is a constant $C_{\beta,n}\geq0$,
independent of $F$, such that
\begin{equation}\label{eq:subcritical-no-remainder}
\log Z(F)-\mean F
\leq\frac{\beta}{2(n+1)!}\E(F)+C_{\beta,n},
\qquad F\in\Lzero.
\end{equation}
Consequently, for $1/2<\beta<\alpha$,
\begin{equation}\label{eq:subcritical-parameter-gap}
\Jf(F)\geq\frac{\alpha-\beta}{2(n+1)!}\E(F)-C_{\beta,n},
\qquad F\in\Lzero.
\end{equation}
In particular, for every $\alpha>1/2$ there is $C_{\alpha,n}\geq0$ such that
\begin{equation}\label{eq:minimizing-sequence-bound}
\Jf(F)\geq\frac{2\alpha-1}{8(n+1)!}\E(F)-C_{\alpha,n},
\qquad F\in\Lzero.
\end{equation}
\end{lemma}
\begin{proof}
We first establish the estimate for the low- and high-degree parts of
$F$ that will be used to prove \eqref{eq:subcritical-no-remainder}.
For each $\gamma>1/2$, we claim that there are an integer
$K=K(n,\gamma)\geq1$ and constants $B_1,B_2\geq0$, independent of $F$,
such that, if $F\in\Lzero$ and $\mean F=0$, the decomposition
\begin{equation}\label{eq:high-low}
F=U+V,\qquad U=\sum_{k=1}^K(\Pi_kF+\overline{\Pi_kF}),\qquad V=F-U
\end{equation}
satisfies
\begin{equation}\label{eq:logmass-high-low}
\log Z(F)\leq\frac\gamma{2(n+1)!}\E(V)+B_1\E(U)^{1/2}+B_2.
\end{equation}

We first choose two separated caps carrying a fixed amount of the full density. Assume that $F\in\Lzero$ is smooth with
$\langle F\rangle_\sigma=0$, and set
\[
\rho=\frac{e^F}{Z(F)}.
\]
Then $\int\rho\dd=1$ and $\int x\rho\dd=0$.
Set
\[
\delta_1=\frac1{2\sqrt Q},\qquad \delta_2=\frac{\delta_1}{8Q}.
\]
The $2Q$ caps $\{\pm x_j\geq\delta_1\}$, $1\leq j\leq Q$, cover the
sphere: otherwise $|x|^2<Q\delta_1^2=1/4$, contrary to $|x|=1$.
Since they cover the sphere and
\[
\int_{S^{2n+1}}\rho\,d\sigma=1,
\]
at least one of them has $\rho$-mass at least $1/(2Q)$.
Hence, for some choice of sign and index, $s=\pm x_j$,
\[
\int_{\{s\geq\delta_1\}}\rho\,d\sigma\geq\frac1{2Q}.
\]
The zero-moment in this coordinate gives
\[
\int_{\{s<0\}}(-s)\rho\dd=\int_{\{s>0\}}s\rho\dd\geq\frac{\delta_1}{2Q}.
\]
Splitting the negative half-space, we obtain
\[
\int_{\{s<0\}}(-s)\rho\,d\sigma
\leq
\delta_2\int_{\{-\delta_2<s<0\}}\rho\,d\sigma
+
\int_{\{s\leq-\delta_2\}}\rho\,d\sigma .
\]
Since
\[
\int_{\{s<0\}}(-s)\rho\,d\sigma\geq\frac{\delta_1}{2Q},
\]
it follows that
\[
\int_{\{s\leq-\delta_2\}}\rho\,d\sigma
\geq
\frac{\delta_1}{2Q}-\delta_2 .
\]
Hence
\[
\int_{\{s\leq-\delta_2\}}\rho\dd
\geq\frac{\delta_1}{2Q}-\delta_2\geq\delta_2.
\]
Choose smooth functions $0\leq\chi_i\leq1$ with
\[
\begin{array}{lll}
\chi_1=1\text{ on }s\geq\delta_1,&\quad&\supp\chi_1\subset\{s>\delta_1/2\},\\
\chi_2=1\text{ on }s\leq-\delta_2,&&\supp\chi_2\subset\{s<-\delta_2/2\}.
\end{array}
\]
Fix these functions once for each of the $2Q$ choices of $s$. This is a finite family independent of $F$. Their supports are disjoint. The sets $A_i=\{\chi_i=1\}$, where the cutoffs equal one, satisfy
\[
\int_{A_i}\rho\dd\geq b_i,\qquad b_1=\frac1{2Q},\quad b_2=\delta_2,
\qquad\chi_1^2+\chi_2^2\leq1.
\]

We next apply the exponential estimate to the localized high-degree part. For any $K$, let $U,V$ be as in \eqref{eq:high-low}. On $A_i$, $\chi_iV=V$, so
\[
b_iZ(F)\leq\int_{A_i}e^{U+V}\dd
\leq e^{\norm U_\infty}\int_{\Sph}e^{\chi_iV}\dd.
\]
Applying \eqref{eq:aux-endpoint} to each $\chi_iV$ gives two
inequalities. After taking logarithms and averaging them, we obtain
\begin{equation}\label{eq:localized-logmass}
\log Z(F)\leq\norm U_\infty
+\frac{a_n(\lambda)}2\sum_{i=1}^2 E_\lambda(\chi_iV)
+\frac12\sum_{i=1}^2\int\chi_iV\dd
+C_{n,\lambda}.
\end{equation}
Here the fixed positive numbers $b_i$ have been absorbed into the
constant $C_{n,\lambda}$. The zero-moment condition is used only for
the normalized density
\[
\frac{e^{U+V}}{Z(F)}.
\]
The densities $e^V$ and $e^{\chi_iV}$ do not necessarily satisfy the zero-moment condition.

To control the lower-order and linear terms, we use the fact that $V$ contains only degrees $k>K$. By \eqref{eq:L-eigenvalue} and \eqref{eq:spectrum},
\begin{equation}\label{eq:etaK}
\eta_K:=\sup_{k>K}\frac{\nu_k^{n}}{\lambda_k}
\leq\frac{[n/2]^{n}}{K+1}\longrightarrow0.
\end{equation}
Consequently,
\[
\ip V{\Ls^{n}V}\leq\eta_K\E(V),\qquad
\norm V_2\leq\lambda_{K+1}^{-1/2}\E(V)^{1/2}.
\]
Since $0\leq\chi_i\leq1$ and the measure has mass one, for every $\theta>0$,
\[
\left|\frac12\sum_{i=1}^2\int\chi_iV\dd\right|
\leq\norm V_2\leq\theta\E(V)+\frac1{4\theta\lambda_{K+1}}.
\]
Using \eqref{eq:two-cutoff-bound} in \eqref{eq:localized-logmass}, we therefore obtain
\begin{equation}\label{eq:coefficient-before-choice}
\log Z(F)\leq\norm U_\infty+
\left[\frac{a_n(\lambda)}2+a_n(\lambda)\epsilon
 +\frac{a_n(\lambda)}2C_{n,\lambda,\epsilon}\eta_K+\theta\right]\E(V)+C.
\end{equation}

We now fix the parameters independently of $F$. Let
\[
d_\gamma=\frac\gamma{2(n+1)!}-\frac1{4(n+1)!}>0.
\]
Choose $\lambda$ sufficiently large so that
\[
a_n(\lambda)/2\leq\frac1{4(n+1)!}+\frac{d_\gamma}{4}.
\]
Next choose $\epsilon>0$ sufficiently small so that
\[
a_n(\lambda)\epsilon\leq\frac{d_\gamma}{4}.
\]
The cutoff family has already been fixed, hence $C_{n,\lambda,\epsilon}$ is now fixed. Choose $K$ large enough that
the term involving $\eta_K$ in the bracket in
\eqref{eq:coefficient-before-choice} is at most $d_\gamma/4$,
using \eqref{eq:etaK}. Finally set
\[
\theta=\frac{d_\gamma}{4}.
\]
The bracket is then at most
\[
\frac{\gamma}{2(n+1)!}.
\]
All parameters are therefore fixed independently of the admissible
function and of any minimizing sequence.

It remains to estimate the low-degree part. The diagonal kernel in \eqref{eq:reproducing-kernel} gives
$\abs{\Pi_kF(\zeta)}\leq d_k^{1/2}\norm{\Pi_kF}_2$. The Cauchy--Schwarz inequality  in the finite sum yields
\[
\norm U_\infty\leq
\left(2\sum_{k=1}^K\frac{d_k}{\lambda_k}\right)^{1/2}\E(U)^{1/2}
=\left(\frac2{n!}\sum_{k=1}^K\frac1k\right)^{1/2}\E(U)^{1/2}.
\]
Indeed $d_k/\lambda_k=1/[n!k]$. Together with the preceding parameter choices, this proves \eqref{eq:logmass-high-low} for smooth $F$.

For general $F\in\Lzero$ with zero mean, apply Lemma~\ref{lem:correction} and subtract the mean from each approximation. They converge in $\Dn$, and their exponential integrals converge by \eqref{eq:exp-continuity}. Their fixed low-degree projections converge in every norm, while their high-degree parts converge in the energy norm. Hence \eqref{eq:logmass-high-low} passes to the limit with the same constants.

To conclude, fix $\beta>1/2$ and use \eqref{eq:logmass-high-low} with $\gamma=\beta$.
The constants $K,B_1,B_2$ are then fixed by $n$ and $\beta$. For every
$t\geq0$, completing the square gives
\[
B_1t\leq\frac{\beta}{2(n+1)!}t^2
+\frac{(n+1)!B_1^2}{2\beta}.
\]
Apply this with $t=\E(U)^{1/2}$ and use
$\E(F)=\E(U)+\E(V)$. We obtain
\[
\begin{split}
\log Z(F)
&\leq\frac{\beta}{2(n+1)!}\E(V)
+\frac{\beta}{2(n+1)!}\E(U)
+B_2+\frac{(n+1)!B_1^2}{2\beta}\\
&=\frac{\beta}{2(n+1)!}\E(F)+C_{\beta,n},
\end{split}
\]
where $C_{\beta,n}=B_2+(n+1)!B_1^2/(2\beta)$ is independent of $F$.
This proves \eqref{eq:subcritical-no-remainder} for $\mean F=0$.
For arbitrary $F\in\Lzero$, apply the same estimate to $F-\mean F$.
Subtracting a constant preserves the zero-moment condition and the
energy, and
$\log Z(F-\mean F)=\log Z(F)-\mean F$.
Thus \eqref{eq:subcritical-no-remainder} holds on all of $\Lzero$.
In particular, no norm of $F$ occurs in the additive remainder.

For $1/2<\beta<\alpha$, we now have
\[
\begin{split}
\Jf(F)
&=\mathcal J_{\beta,n}(F)
+\frac{\alpha-\beta}{2(n+1)!}\E(F)\\
&\geq\frac{\alpha-\beta}{2(n+1)!}\E(F)-C_{\beta,n}.
\end{split}
\]
This proves \eqref{eq:subcritical-parameter-gap}. Taking
$\beta=(\alpha+1/2)/2$ yields
\[
\frac{\alpha-\beta}{2(n+1)!}
=\frac{2\alpha-1}{8(n+1)!}>0.
\]
Renaming the constant gives \eqref{eq:minimizing-sequence-bound}.
\end{proof}

Next, we prove the existence of extremals of the subcritical Moser--Onofri inequality under zero-moment constraints.

\begin{proposition}\label{prop:existence}
For every $\alpha>1/2$, the functional $\Jf$ attains its infimum on $\Lzero$. Every minimizing sequence, after subtracting the mean from each term, has a subsequence converging strongly in $\Dn$ to a minimizer.
\end{proposition}
\begin{proof}
We prove the proposition by the direct method. Fix $\alpha>1/2$ and
choose $1/2<\beta<\alpha$. By Lemma~\ref{lem:subcritical},
\[
\Jf(F)\geq\frac{\alpha-\beta}{2(n+1)!}\E(F)-C_{\beta,n},
\qquad F\in\Lzero.
\]
Constants belong to the admissible class and give functional value zero.
Thus $-\infty<c_\alpha:=\inf_{\Lzero}\Jf\leq0$.
Let $F_j$ be a minimizing sequence. Replacing each $F_j$ by
$F_j-\langle F_j\rangle_\sigma$ preserves its zero moment and its
functional value, and gives $\mean{F_j}=0$.
After discarding finitely many terms, we have $\Jf(F_j)\leq1$.
The positive difference $\alpha-\beta$ therefore gives the uniform bound
\begin{equation}\label{eq:minimizing-energy-bound}
\E(F_j)\leq
\frac{2(n+1)!}{\alpha-\beta}\bigl(1+C_{\beta,n}\bigr)
=:M<\infty.
\end{equation}
Both $\beta$ and $C_{\beta,n}$ have been fixed independently of the
sequence. Hence the sequence with zero mean is bounded in $\Dn$.
Weak compactness in the zero-mean Hilbert space gives, after passing to a subsequence, $F_j\rightharpoonup F$ in $\Dn$ and $\mean F=0$.

To pass to the limit in the nonlinear term, we first establish strong convergence in $L^2$. For a fixed $K$, $T_KF_j\to T_KF$ in the finite-dimensional space of degrees at most $K$. Also
\[
\norm{F_j-T_KF_j}_2^2\leq\frac{M}{\lambda_{K+1}},\qquad
\norm{F-T_KF}_2^2\leq\frac{M}{\lambda_{K+1}}.
\]
The second estimate follows from weak lower semicontinuity of the form. First let $j\to\infty$ at fixed $K$, then let $K\to\infty$; this proves $F_j\to F$ in $L^2$. A further subsequence converges almost everywhere.

Fix $p>1$. The known inequality \eqref{eq:BFM}, applied to $pF_j$, gives
\[
\int e^{pF_j}\dd\leq\exp\left(\frac{p^2M}{2(n+1)!}\right).
\]
For a measurable set $A\subset\Sph$, H\"older's inequality gives
\[
\int_A e^{F_j}\dd\leq
\exp\left(\frac{pM}{2(n+1)!}\right)\sigma(A)^{1-1/p},
\]
uniformly in $j$. Thus $\{e^{F_j}\}$ is uniformly integrable. Since $F_j\to F$ almost everywhere, Vitali's theorem gives $e^{F_j}\to e^F$ in $L^1$. Together with $\abs x=1$, this implies
\begin{equation}\label{eq:mass-moment-limit}
Z(F_j)\to Z(F),\qquad M(F)=\lim_jM(F_j)=0.
\end{equation}
Jensen's inequality and $\mean F=0$ give $Z(F)\geq1$; finiteness follows from \eqref{eq:BFM}. Therefore the logarithms of the masses also converge. The limit belongs to $\Lzero$, and weak lower semicontinuity gives
\[
\Jf(F)\leq\liminf_j\Jf(F_j)=c_\alpha.
\]
By the definition of the infimum, equality holds, so $F$ is a minimizer.

Finally, the defining identity for $\Jf$ and \eqref{eq:mass-moment-limit} imply
\[
\E(F_j)=\frac{2(n+1)!}{\alpha}\bigl(\Jf(F_j)+\log Z(F_j)\bigr)
\longrightarrow\frac{2(n+1)!}{\alpha}\bigl(\Jf(F)+\log Z(F)\bigr)=\E(F).
\]
Since both means are zero, weak convergence also gives
\[
\norm{F_j-F}_{\Dn}^2
=\E(F_j)+\E(F)-2\ip{\A F}{F_j}\longrightarrow0.
\]
The convergence $\ip{\A F}{F_j}\to\E(F)$ follows from weak convergence through the continuous pairing associated with the quadratic form. This completes the proof of strong convergence for a subsequence of every normalized minimizing sequence.
\end{proof}

Finally, we derive the Euler--Lagrange equation, establish the smoothness of extremals, and eliminate the moment multiplier. These arguments apply to constrained critical points in the energy space, whose definition we specify first.

A function $F\in\Lzero$ is called a \emph{constrained critical point of $\Jf$ in $\Dn$ under the zero-moment condition} if
\[
D\Jf(F)[H]=0
\quad\text{for every }H\in\Dn\text{ such that }
\int_{\Sph}xHe^F\dd=0.
\]
The derivative is taken with respect to the real Hilbert-space norm of $\Dn$. By Lemma~\ref{lem:correction}, this is equivalent to
\[
\left.\frac{d}{dt}\Jf(F_t)\right|_{t=0}=0
\]
for every $C^1$ curve $t\mapsto F_t$ in $\Dn$ satisfying $F_0=F$ and $M(F_t)=0$. In particular, every minimizer of $\Jf$ on $\Lzero$ is a constrained critical point in this sense.

Since adding a constant preserves both $\Jf$ and the zero-moment condition, we normalize $Z(F)=1$. With $\rho=e^F$, the tangent space to the constraints $Z(F)=1$ and $M(F)=0$ is
\[
\mathcal T_F=\left\{H\in\Dn:\int H\rho\dd=0,\ \int xH\rho\dd=0\right\}.
\]
Indeed, $1,x_1,\ldots,x_Q$ belong to $\Dn$, and the derivative of $(Z,M)$ on their span has the invertible matrix \eqref{eq:constraint-matrix}. Lemma~\ref{lem:correction} therefore realizes every vector in $\mathcal T_F$ as the velocity of a curve satisfying both constraints exactly. If $H\in\Dn$ satisfies $\int xH\rho\dd=0$, subtracting the constant $\int H\rho\dd$ places it in $\mathcal T_F$ and leaves $D\Jf(F)[H]$ unchanged. Consequently, the definition above is equivalent to $D\Jf(F)[H]=0$ for all $H\in\mathcal T_F$ under the mass normalization.

We now derive the weak Euler--Lagrange equation and prove smoothness for $\alpha>0$. We then eliminate the moment multiplier for $0<\alpha\leq1$.

\begin{proposition}\label{prop:Euler}
Let $\alpha>0$, and let $F\in\Lzero$ be a constrained critical point, normalized by $Z(F)=1$. Then $F$ is smooth and there is $\ell\in\R^Q$ such that
\begin{equation}\label{eq:Euler-multiplier}
\frac\alpha{(n+1)!}\A F=\PP\bigl((1+x\cdot\ell)e^F-1\bigr).
\end{equation}
If $x\cdot a$ is the real degree-one component of $F$ and $C=\int xx^Te^F\dd$, then
\begin{equation}\label{eq:degree-one}
\alpha a=QC\ell.
\end{equation}
\end{proposition}
\begin{proof}
By Lemma~\ref{lem:correction}, the first variation is well defined on the energy space and satisfies, for every $H\in\Dn$,
\begin{equation}\label{eq:first-variation}
D\Jf(F)[H]=\frac\alpha{(n+1)!}\ip{\A F}H+\int H\dd-\int H\rho\dd.
\end{equation}
The matrix $C$ is positive definite by \eqref{eq:constraint-matrix}. Define
\[
H_0=H-\int H\rho\dd-x\cdot C^{-1}\int xH\rho\dd.
\]
The identities $\int\rho\dd=1$ and $\int x\rho\dd=0$ show directly that $H_0\in\mathcal T_F$. Let $q_i=D\Jf(F)[x_i]$. Since $D\Jf(F)[H_0]=0$ and $D\Jf(F)[1]=0$, we obtain
\[
D\Jf(F)[H]=q^TC^{-1}\int xH\rho\dd
=\ell\cdot\int xH\rho\dd,\qquad \ell=C^{-1}q.
\]
Together with \eqref{eq:first-variation}, we obtain
\begin{equation}\label{eq:weak-Euler}
\frac\alpha{(n+1)!}\ip{\A F}H
=\int H\bigl((1+x\cdot\ell)\rho-1\bigr)\dd
\qquad(H\in\Dn),
\end{equation}
Here $\ip{\A F}H$ is the continuous bilinear pairing associated with the energy form. We interpret \eqref{eq:Euler-multiplier} weakly by \eqref{eq:weak-Euler}, with test functions in $\Dn$. Since these tests are real pluriharmonic, projecting the right-hand side leaves the pairing unchanged. For any smooth real test function $\varphi$, using $H=\PP\varphi$ also gives \eqref{eq:Euler-multiplier} as a distributional identity. Conversely, a function satisfying the mass and moment constraints and \eqref{eq:weak-Euler} has $D\Jf(F)[H]=0$ for every $H\in\mathcal T_F$. Thus the weak multiplier equation is equivalent to constrained criticality in $\Dn$.

Testing \eqref{eq:first-variation} with $x_i$ gives $q_i=\alpha a_i/Q$. Indeed, the last two integrals vanish, $\A x_i=(n+1)!x_i$, and
\begin{equation}\label{eq:coordinate-normalization}
\int x_ix_j\dd=\frac{\delta_{ij}}Q,\qquad
\int Fx_i\dd=\frac{a_i}Q.
\end{equation}
Since $q=C\ell$, this proves \eqref{eq:degree-one} as well.

To establish regularity, we use the spectral decomposition of $\A$. Set
\[
G=(1+x\cdot\ell)e^F-1.
\]
It has zero mean because $Z(F)=1$ and $\int_{\Sph}xe^F\dd=0$. By \eqref{eq:exponential-bound}, $G\in L^2$ (in fact $G\in L^p$ for every finite $p$). In this proof, $W^{s,2}$ abbreviates $W^{s,2}(\Sph)$, the ordinary $L^2$ Sobolev space on the sphere of real dimension $2n+1$. If $P_m$ is the projection onto ordinary spherical harmonics of degree $m$, use the spectral norm
\[
\norm u_{W^{s,2}}^2=\sum_{m\geq0}b_m^s\norm{P_mu}_2^2,
\qquad b_m=1+m(m+2n).
\]
The decomposition \eqref{eq:bispherical} refines the ordinary one by $m=j+k$. Hence $\PP$ commutes with $P_m$ and is a contraction in every $W^{s,2}$.

Write $\pi_{\mathcal P,k}G=\Pi_kG+\overline{\Pi_kG}$ for the degree-$k$ component of $\PP G$, $k\geq1$. Testing the weak equation in each such space determines its coefficients:
\begin{equation}\label{eq:inverse-series}
F-\mean F=\frac{(n+1)!}{\alpha}\sum_{k\geq1}\lambda_k^{-1}\pi_{\mathcal P,k}G.
\end{equation}
For $k\geq1$, the estimates
\[
b_k\leq2(n+1)k^2,
\qquad
\lambda_k\geq k^{n+1}
\]
imply
\[
\frac{b_k^{n+1}}{\lambda_k^2}
\leq(2(n+1))^{n+1}.
\]
For a positive integer $L$, denote the finite sum in \eqref{eq:inverse-series} by
\[
W_L=\frac{(n+1)!}{\alpha}\sum_{k=1}^L\lambda_k^{-1}\pi_{\mathcal P,k}G.
\]
If $G\in W^{s,2}$ with $s\geq0$, orthogonality and the preceding comparison give
\begin{align}
\norm{W_L}_{W^{s+n+1,2}}^2
&=\left(\frac{(n+1)!}{\alpha}\right)^2
\sum_{k=1}^L\frac{b_k^{s+n+1}}{\lambda_k^2}
\norm{\pi_{\mathcal P,k}G}_2^2\notag\\
&\leq\left(\frac{(n+1)!}{\alpha}\right)^2(2(n+1))^{n+1}
\sum_{k=1}^L b_k^s\norm{\pi_{\mathcal P,k}G}_2^2\notag\\
&\leq\left(\frac{(n+1)!}{\alpha}\right)^2(2(n+1))^{n+1}\norm G_{W^{s,2}}^2.
\label{eq:finite-spectral-gain}
\end{align}
In the last step, $\pi_{\mathcal P,k}G$ is an orthogonal component of $P_kG$, and different total degrees are orthogonal. The same estimate for $W_{L_2}-W_{L_1}$ has the tail of the convergent $W^{s,2}$ series of $G$ on its right. It follows that $W_L$ is a Cauchy sequence in $W^{s+n+1,2}$, which justifies convergence of the inverse series in the stronger Sobolev space. Its limit has exactly the nonconstant coefficients of $F$, determined by the weak equation, and has mean zero. Therefore it equals $F-\mean F$. Letting $L\to\infty$ in \eqref{eq:finite-spectral-gain} proves
\begin{equation}\label{eq:spectral-gain}
\norm{F-\mean F}_{W^{s+n+1,2}}
\leq\frac{(n+1)!}{\alpha}(2(n+1))^{(n+1)/2}\norm G_{W^{s,2}}
\qquad(G\in W^{s,2},\ s\geq0).
\end{equation}
In particular, $G\in L^2$ gives $F\in W^{n+1,2}$. This initial gain of regularity allows us to apply the Sobolev composition estimates in the next step.

We can now iterate this gain using Sobolev composition. The usual Sobolev embedding on a compact manifold gives $W^{s,2}\subset C^0$ for $s>(2n+1)/2$; see \cite{TaylorI}. For integer $s$ in this range, the standard product and smooth-composition estimates give $uv\in W^{s,2}$ for $u,v\in W^{s,2}$, and $e^u\in W^{s,2}$ for real $u\in W^{s,2}$. The product and composition estimates are given in \cite{TaylorIII}. A finite coordinate cover and a smooth partition of unity give these estimates on the sphere. In the composition statement one may first apply the estimate to $e^u-1$ and then add the constant function.

Since $(n+1)>(2n+1)/2$, $F\in W^{n+1,2}$ implies $G\in W^{n+1,2}$. Estimate \eqref{eq:spectral-gain} gives $F\in W^{2(n+1),2}$. If $F\in W^{j(n+1),2}$ for some integer $j\geq1$, the same estimates
yield $G\in W^{j(n+1),2}$, and \eqref{eq:spectral-gain} gives $F\in W^{(j+1)(n+1),2}$.
Iterating this argument gives $F\in W^{j(n+1),2}$ for every integer $j\geq1$.
For any nonnegative integer $r$, choose $j$ such that
\[
j(n+1)>\frac{2n+1}{2}+r .
\]
The Sobolev embedding theorem implies
\[
F\in C^r.
\]
Since $r$ is arbitrary, $F$ is smooth.
\end{proof}

The smoothness established in Proposition~\ref{prop:Euler} allows us to differentiate the CR conformal action. We use these variations to show that the moment multiplier vanishes for $0<\alpha\leq1$. Let $\cstr$ be the standard complex structure on $\R^Q$, so $\cstr^T=-\cstr$ and $\cstr^2=-I$. For $v\in\R^Q$, write $x_v=x\cdot v$. The infinitesimal CR automorphism associated with the holomorphic vector field $z\mapsto a-(z\cdot\bar a)z$, where $a$ corresponds to $v$, is the real tangential field
\begin{equation}\label{eq:real-conformal-field}
X_v(x)=v-x_vx-x_{\cstr v}\cstr x.
\end{equation}
For example, in the first real coordinate direction it is obtained by differentiating
\[
\tau_t(z)=\left(\frac{z_1+\tanh t}{1+z_1\tanh t},
\frac{\operatorname{sech}t}{1+z_1\tanh t}(z_2,\ldots,z_{n+1})\right)
\]
at $t=0$. Unitary conjugation and rescaling give arbitrary $v$. These are ball automorphisms, and their boundary affine action is \eqref{eq:affine-action}.

We will use the following two identities:
\begin{equation}\label{eq:field-identities}
\Div X_v=-Qx_v,\qquad
X_vx_w=v\cdot w-x_vx_w+x_{\cstr v}x_{\cstr w}.
\end{equation}
For completeness, we verify these identities directly. The field $v-x_vx$ is the spherical gradient of $x_v$ and has divergence $-(Q-1)x_v$. The rotation field $x\mapsto\cstr x$ has zero divergence, and
$(\cstr x)(x_{\cstr v})=(\cstr x)\cdot\cstr v=x_v$. The last term of \eqref{eq:real-conformal-field} therefore has divergence $-x_v$. This gives the first identity. For the second, take the Euclidean scalar product of \eqref{eq:real-conformal-field} with $w$ and use $w\cdot\cstr x=-x_{\cstr w}$.

For a smooth $F\in\PR$, differentiation of its affine action gives the smooth real pluriharmonic function
\[
Y_v=X_vF-Qx_v.
\]
The invariance of the pluriharmonic space under the affine action, quoted in \eqref{eq:affine-action}, justifies this assertion. If $\rho=e^F$, then
\begin{equation}\label{eq:density-derivative}
Y_v\rho=\Div(\rho X_v),\qquad\int Y_v\rho\dd=0.
\end{equation}
The last identity is the divergence theorem on the closed sphere.

\begin{proposition}\label{prop:eliminate}
Let $0<\alpha\leq1$. For every constrained critical point in Proposition~\ref{prop:Euler}, the multiplier $\ell$ vanishes. Thus
\begin{equation}\label{eq:Euler-free}
\frac\alpha{(n+1)!}\A F=\PP(e^F-1),\qquad Z(F)=1,
\end{equation}
and the degree-one component of $F$ also vanishes.
\end{proposition}
\begin{proof}
Proposition~\ref{prop:Euler} gives smoothness of $F$. We may therefore use the infinitesimal conformal variations as test functions in \eqref{eq:weak-Euler}. Define
\[
C=\int xx^T\rho\dd,\qquad H=I-C+\cstr^TC\cstr.
\]
For any $w\in\R^Q$, integration by parts, \eqref{eq:density-derivative}, and \eqref{eq:field-identities} give
\begin{align}
w\cdot DM(F)[Y_v]
&=\int x_w\Div(\rho X_v)\dd=-\int (X_vx_w)\rho\dd\notag\\
&=-v\cdot w+w^TCv-w^T\cstr^TC\cstr v
=-w^THv.
\label{eq:moment-derivative}
\end{align}
Thus $DM(F)[Y_v]=-Hv$. The matrix $H$ is positive definite. Indeed, for $v\ne0$,
\begin{equation}\label{eq:H-positive}
v^THv=\int\bigl(\abs v^2-x_v^2+x_{\cstr v}^2\bigr)\rho\dd>0.
\end{equation}
The integrand is nonnegative by $\abs{x_v}\leq\abs v$; it can vanish only at the two points $x=\pm v/\abs v$. The density is positive, and these points have zero surface measure.

Next use the exact identity
\[
\Jf=\alpha\Jone+(1-\alpha)(\mean F-\log Z(F)).
\]
Since both $\Jone$ and the exponential integral are invariant under the CR conformal action, differentiation gives
\[
D\Jf(F)[Y_v]=(1-\alpha)\int Y_v\dd.
\]
Using \eqref{eq:field-identities}, integration by parts, and
\eqref{eq:coordinate-normalization}, we obtain
\[
\int Y_v\,d\sigma
=
\int X_vF\,d\sigma
=
Q\int Fx_v\,d\sigma
=
a\cdot v .
\]
On the other hand, the multiplier relation proved in \eqref{eq:weak-Euler} is valid for every test in $\Dn$, not only tangent tests. Together with \eqref{eq:moment-derivative}, it gives
\[
D\Jf(F)[Y_v]=\ell\cdot DM(F)[Y_v]=-\ell^THv.
\]
Combining the last three identities and using $\alpha a=QC\ell$ yields
\begin{equation}\label{eq:multiplier-matrix}
\left(H+\frac{Q(1-\alpha)}\alpha C\right)\ell=0.
\end{equation}
The matrix $C$ is positive definite, $H$ is positive definite by \eqref{eq:H-positive}, and the coefficient of $C$ is nonnegative. Hence the matrix in \eqref{eq:multiplier-matrix} is invertible and $\ell=0$. Formula \eqref{eq:degree-one} then gives $a=0$. We emphasize that $Y_v$ need not preserve the moment; the calculation uses the full multiplier relation.
\end{proof}

\section{Classification of extremals for Moser-Onofri inequality under zero-moment constraint}
In this section, we show that every minimizer of $\Jf$ on $\Lzero$ is constant for $1/2<\alpha\leq1$. This classifies the extremals of the subcritical Moser--Onofri inequality under zero-moment constraints obtained in Section~3 and will also help to determine the sharp Moser-Onofri inequality under zero-moment constraint in Theorem 1.1. We first derive a holomorphic moment identity from the Euler--Lagrange equation. We then use the zero-moment condition and the summation identity \eqref{eq:shift-sum} on the Hardy space $H^2(\Ball)$ to prove Theorem~\ref{thm:moment-estimate}. Comparing this inequality with the moment identity, and using its equality case, gives the classification of extremals .

We begin with a smooth centered solution $F$ of \eqref{eq:Euler-free} and write
\begin{equation}\label{eq:holomorphic-reduction}
f=\Pi F-\frac12\mean F,\qquad F=f+\bar f,\qquad h=e^f,\qquad \rho=e^F=\abs h^2.
\end{equation}
The projection $\Pi$ preserves each ordinary spherical Sobolev space: on every total degree it is an orthogonal projection. Hence $f$ is smooth on $\Sph$. Its expansion consists only of homogeneous holomorphic polynomials, so it has a holomorphic extension into $\Ball$, as in the Hardy-space description in \cite{Rudin,Zhu}. More explicitly, the harmonic extension of its boundary expansion is $f(r\zeta)=f(0)+\sum_{k\geq1}r^kf_k(\zeta)$. This is the Poisson extension of smooth data and is smooth on the closed ball. The series is holomorphic in the interior, and thus agrees there with the Hardy extension. Its constant term is $f(0)=\mean F/2\in\R$. Therefore $h=e^f$ is holomorphic, nonzero, and smooth up to the boundary, and
\[
\norm h_2=1,\qquad
\ip h{S_jh}=\int_{\Sph}\zeta_j\rho\dd=0\quad(1\leq j\leq n+1).
\]
Write $f=f(0)+\sum_{k\geq1}f_k$, where $f_k\in\mathcal H_{k,0}$. The holomorphic degree-$k$ part of \eqref{eq:Euler-free} is
\begin{equation}\label{eq:holomorphic-Euler}
\Pi_k\rho=\frac{\alpha\lambda_k}{(n+1)!}f_k\qquad(k\geq1).
\end{equation}
In particular $f_1=0$.

For any nonnegative integrable density $\rho$ of mass one, its holomorphic moments define the nonnegative numbers
\begin{align}
s_k&=\iint_{\Sph\times\Sph}(\zeta\cdot\bar\eta)^k
\rho(\zeta)\rho(\eta)\,d\sigma_\zeta\,d\sigma_\eta\notag\\
&=\sum_{\abs\beta=k}\frac{k!}{\beta!}
\left|\int_{\Sph}\zeta^\beta\rho(\zeta)\,d\sigma_\zeta\right|^2,
\qquad k\geq1.
\label{eq:sk}
\end{align}
The second equality follows by the finite multinomial expansion. Both sides are defined even when $\rho$ is only in $L^1$, since the monomials and the kernel in the first line are bounded by one. When $\rho\in L^2$, the reproducing kernel \eqref{eq:reproducing-kernel} also gives
\begin{equation}\label{eq:projected-moment}
\norm{\Pi_k\rho}_2^2=d_ks_k.
\end{equation}
Indeed, self-adjointness and idempotence of $\Pi_k$ give
$\norm{\Pi_k\rho}_2^2=\ip\rho{\Pi_k\rho}$, and substituting its kernel gives precisely $d_ks_k$. The zero moment is equivalent to $s_1=0$. For $\rho=\abs h^2$, write
\[
\Scal(h)=\sum_{k\geq2}s_k,
\]
allowing initially the value $+\infty$. The Euler--Lagrange equation relates these moments to $\Rcal(h)$ as follows.

\begin{proposition}\label{prop:Euler-moment}
For $f,h,\rho$ obtained from a smooth solution as in \eqref{eq:holomorphic-reduction}--\eqref{eq:holomorphic-Euler}, the series $\Scal(h)$ is finite and
\begin{equation}\label{eq:Euler-moment}
\Rcal(h)=\frac{n+1}\alpha\Scal(h).
\end{equation}
\end{proposition}
\begin{proof}
The holomorphic chain rule gives $Rh=hRf$. Since $h$ is smooth,
\begin{equation}\label{eq:radial-chain}
\Rcal(h)=\ip h{Rh}=\int_{\Sph}\rho Rf\dd<\infty.
\end{equation}
For each $k\geq1$, the reality of $\rho$ and the inner-product convention give
\begin{align}
\int\rho f_k\dd
&=\ip\rho{f_k}=\ip{\Pi_k\rho}{f_k}
=\frac{\alpha\lambda_k}{(n+1)!}\norm{f_k}_2^2
=\frac{(n+1)!}{\alpha\lambda_k}\norm{\Pi_k\rho}_2^2.
\label{eq:complex-moment-pairing}
\end{align}
In the second equality $f_k$ lies in the range of $\Pi_k$. In the third, the coefficient in \eqref{eq:holomorphic-Euler} is real. Thus the inner-product convention gives a real, nonnegative value for this integral.

First perform the summation with $f^{(L)}=f(0)+\sum_{k=1}^Lf_k$. By \eqref{eq:complex-moment-pairing} and \eqref{eq:projected-moment},
\begin{equation}\label{eq:finite-Euler-moment}
\int\rho Rf^{(L)}\dd
=\sum_{k=1}^L\frac{k(n+1)!}{\alpha\lambda_k}\norm{\Pi_k\rho}_2^2
=\sum_{k=1}^L\frac{k(n+1)!d_k}{\alpha\lambda_k}s_k.
\end{equation}
Using the explicit values of the eigenvalue and the dimension, we have
\[
\frac{k(n+1)!d_k}{\lambda_k}
=k(n+1)!\frac{(k+n)!}{k!n!}\frac{(k-1)!}{(k+n)!}=(n+1).
\]
The smoothness of $f$ implies $Rf^{(L)}\to Rf$ in $L^2$. Since $\rho\in L^2$, the left side of \eqref{eq:finite-Euler-moment} converges to $\Rcal(h)$, which is \eqref{eq:radial-chain}. The right side is an increasing sum of nonnegative terms. Hence its limit is finite, and
$\Rcal(h)=((n+1)/\alpha)\sum_{k\geq1}s_k$. Finally $s_1=0$, proving \eqref{eq:Euler-moment}.
\end{proof}

To compare $\Scal(h)$ with $\Rcal(h)$, we first prove the summation identity \eqref{eq:shift-sum} on the Hardy space $H^2(\Ball)$. For a multi-index $\beta\in\N_0^{n+1}$, define
\[
S^\beta=S_1^{\beta_1}\cdots S_{n+1}^{\beta_{n+1}},\qquad
S^{*\beta}=(S_1^*)^{\beta_1}\cdots(S_{n+1}^*)^{\beta_{n+1}},\qquad
c_\beta=\frac{(\abs\beta)!}{\beta!}.
\]
The operators $S_j$ commute, and their adjoints commute as well. We use the falling factorial
$(a)_k=a(a-1)\cdots(a-k+1)$, with $(a)_0=1$.

\begin{lemma}\label{lem:shift-sum}
For $h\in H^2(\Ball)$ with $\Rcal(h)<\infty$,
\begin{equation}\label{eq:shift-sum}
\sum_{\beta\ne0}c_\beta\norm{S^{*\beta}h}_2^2=\frac1{n+1}\Rcal(h).
\end{equation}
\end{lemma}
\begin{proof}
We first compute the sum on a fixed homogeneous degree. Let $\abs\gamma=m$ and $\abs\beta=k\leq m$. Repeated application of \eqref{eq:one-shift} gives
\begin{equation}\label{eq:multi-shift-weight}
\norm{S^{*\beta}e_\gamma}_2^2
=\frac{\prod_{j=1}^{n+1}(\gamma_j)_{\beta_j}}{(n+m)_k},
\end{equation}
with value zero unless $\beta_j\leq\gamma_j$ for every $j$. The denominator records the consecutive total degrees $m,m-1,\ldots,m-k+1$; the numerator records the factors obtained from the $\beta_j$ applications of $S_j^*$ in each coordinate. The elementary falling-factorial multinomial identity gives
\[
\sum_{\abs\beta=k}\frac{k!}{\beta!}\prod_{j=1}^{n+1}(\gamma_j)_{\beta_j}=(m)_k.
\]
Thus the sum of \eqref{eq:multi-shift-weight} with weights $c_\beta$ is $(m)_k/(n+m)_k$.

For an arbitrary $h_m=\sum_{\abs\gamma=m}a_\gamma e_\gamma$, fix $\beta$. Nonzero images of different $e_\gamma$ are multiples of different $e_{\gamma-\beta}$, so they are orthogonal. Consequently the preceding identity extends to
\begin{equation}\label{eq:homogeneous-shift-sum}
\sum_{\abs\beta=k}c_\beta\norm{S^{*\beta}h_m}_2^2
=\frac{(m)_k}{(n+m)_k}\norm{h_m}_2^2.
\end{equation}

We next sum over $k=\abs\beta$. For $1\leq k\leq m$, direct cancellation gives
\[
\frac{(m)_k}{(n+m)_k}
=\frac m{n+1}\left[\frac{(m-1)_{k-1}}{(n+m)_{k-1}}
-\frac{(m-1)_k}{(n+m)_k}\right].
\]
To check the identity, factor the first fraction inside the brackets; the remaining factor is
$1-(m-k)/(n+m+1-k)=(n+1)/(n+m+1-k)$. All denominators are positive. The bracketed fractions start at $1$ when $k-1=0$ and end at $0$ when $k=m$. Therefore
\begin{equation}\label{eq:sum-lengths}
\sum_{k=1}^m\frac{(m)_k}{(n+m)_k}=\frac m{n+1}.
\end{equation}
Combining \eqref{eq:homogeneous-shift-sum} and \eqref{eq:sum-lengths} proves the summation identity \eqref{eq:shift-sum} on the Hardy space $H^2(\Ball)$ for a single homogeneous function.

To pass to a general $h$, consider its spectral truncations $h^{(M)}=\sum_{m=0}^Mh_m$. For a fixed $\beta$, the nonzero functions $S^{*\beta}h_m$ have distinct homogeneous degrees $m-\abs\beta$. Hence
\[
\norm{S^{*\beta}h^{(M)}}_2^2
=\sum_{m=\abs\beta}^M\norm{S^{*\beta}h_m}_2^2.
\]
All sums are finite at this stage. Summing first in $\abs\beta=k$ and then in $k$, the preceding identities give
\[
\sum_{\beta\ne0}c_\beta\norm{S^{*\beta}h^{(M)}}_2^2
=\frac1{n+1}\sum_{m=1}^M m\norm{h_m}_2^2.
\]
For each fixed $\beta$, the squared norm on the left increases with $M$ by the orthogonal decomposition just written. It converges to $\norm{S^{*\beta}h}_2^2$, since $S^{*\beta}$ is bounded and $h^{(M)}\to h$ in $H^2$. Monotone convergence over the multi-indices now proves the summation identity \eqref{eq:shift-sum} on the Hardy space $H^2(\Ball)$.
\end{proof}

We now use the summation identity \eqref{eq:shift-sum} on the Hardy space $H^2(\Ball)$ and the zero-moment condition to bound $\Scal(h)$, the weighted sum of the squared moments of $\abs h^2$, in terms of $\Rcal(h)$.

\begin{theorem}\label{thm:moment-estimate}
Let $n\geq1$. Suppose that $h\in H^2(\Ball)$ satisfies
\[
\norm h_2=1,\qquad \ip h{S_jh}=0\quad(1\leq j\leq n+1),\qquad
\Rcal(h)<\infty.
\]
Define $s_k$ by \eqref{eq:sk} with $\rho=\abs h^2$. Then $\Scal(h)<\infty$ and
\begin{equation}\label{eq:moment-estimate}
\Rcal(h)\geq2(n+1)\Scal(h).
\end{equation}
The coefficient $2n+2$ cannot be increased. Equality holds if and only if $h$ is constant.
\end{theorem}
\begin{proof}
We decompose $S^{*\beta}h$ into its orthogonal projection onto $\C h$ and the remainder. For $\beta\ne0$, set
\begin{equation}\label{eq:moment-residual}
m_\beta=\ip h{S^\beta h}=\int\zeta^\beta\abs h^2\dd,
\qquad r_\beta=S^{*\beta}h-\overline{m_\beta}\,h.
\end{equation}
The moments are well defined because $\abs h^2\in L^1$. Thus the theorem does not require an $L^2$ density. The adjoint identity, together with our convention that the inner product is conjugate-linear in its first argument, gives
\[
\ip h{S^{*\beta}h}=\ip{S^\beta h}h=\overline{m_\beta},
\qquad \ip h{r_\beta}=0.
\]
Since $\norm h_2=1$, we obtain
\begin{equation}\label{eq:shift-decomposition}
\norm{S^{*\beta}h}_2^2=\abs{m_\beta}^2+\norm{r_\beta}_2^2.
\end{equation}
By the summation identity \eqref{eq:shift-sum} on the Hardy space $H^2(\Ball)$, both weighted sums on the right are finite. Since $m_{\varepsilon_j}=0$ for $1\leq j\leq n+1$, \eqref{eq:sk} and \eqref{eq:shift-decomposition} give
\begin{equation}\label{eq:S-T}
\Scal(h)=\sum_{\abs\beta\geq2}c_\beta\abs{m_\beta}^2,\qquad
\Tcal(h)=\sum_{\beta\ne0}c_\beta\norm{r_\beta}_2^2,\qquad
\frac1{n+1}\Rcal(h)=\Scal(h)+\Tcal(h).
\end{equation}

The zero-moment condition allows us to express the higher moments in terms of these remainders. For $\beta\ne0$ and $1\leq j\leq n+1$, the adjoint identity gives
\[
m_{\beta+\varepsilon_j}=\ip{S^{*\beta}h}{S_jh}.
\]
Inserting \eqref{eq:moment-residual} removes the component parallel to $h$, because $\ip h{S_jh}=0$. Hence
\[
m_{\beta+\varepsilon_j}=\ip{r_\beta}{S_jh}.
\]
For $\abs\gamma=k\geq2$ and $\gamma_j>0$,
\[
c_{\gamma-\varepsilon_j}
=\frac{(k-1)!}{(\gamma_j-1)!\prod_{i\ne j}\gamma_i!}
=\frac{\gamma_j}{k}c_\gamma.
\]
Therefore $\sum_{j:\gamma_j>0}c_{\gamma-\varepsilon_j}=c_\gamma$. Reindexing the nonnegative series by $\gamma=\beta+\varepsilon_j$, which is justified by Tonelli's theorem, gives
\begin{align}
\Scal(h)
&=\sum_{\beta\ne0}c_\beta\sum_{j=1}^{n+1}\abs{\ip{r_\beta}{S_jh}}^2\notag\\
&\leq\sum_{\beta\ne0}c_\beta\norm{r_\beta}_2^2
\sum_{j=1}^{n+1}\norm{S_jh}_2^2=\Tcal(h).
\label{eq:S-bounded-by-T}
\end{align}
Here each estimate is Cauchy--Schwarz, and
$\sum_j\norm{S_jh}_2^2=\int\sum_j\abs{\zeta_j}^2\abs h^2\dd=1$.
Combining \eqref{eq:S-bounded-by-T} with \eqref{eq:S-T} proves \eqref{eq:moment-estimate}.

We next determine when equality holds. For every pair $(\beta,j)$ with $\beta\ne0$, define
\[
\Delta_{\beta,j}=\norm{r_\beta}_2^2\norm{S_jh}_2^2
-\abs{\ip{r_\beta}{S_jh}}^2\geq0.
\]
For each positive integer $L$, the reindexing $\gamma=\beta+\varepsilon_j$ used above gives the finite identity
\[
\sum_{1\leq|\beta|\leq L}c_\beta\sum_{j=1}^{n+1}\Delta_{\beta,j}
=\sum_{1\leq|\beta|\leq L}c_\beta\norm{r_\beta}_2^2
 -\sum_{2\leq|\gamma|\leq L+1}c_\gamma|m_\gamma|^2.
\]
The upper degree in the last sum is $L+1$ because one coordinate is added to $\beta$. Both series on the right converge to finite values by \eqref{eq:S-T}. Letting $L\to\infty$ therefore gives
\[
\Tcal(h)-\Scal(h)
=\sum_{\beta\ne0}c_\beta\sum_{j=1}^{n+1}\Delta_{\beta,j}.
\]
Equality in \eqref{eq:moment-estimate} means $\Tcal(h)=\Scal(h)$. Every weight $c_\beta$ is positive; hence every $\Delta_{\beta,j}$ is zero.

If some $r_\beta$ were nonzero, equality in each Cauchy--Schwarz estimate would force all the functions $S_jh=z_jh$ to belong to the one-dimensional space $\C r_\beta$. But $z_1h$ and $z_2h$ are linearly independent. Indeed, a relation $(a_1z_1+a_2z_2)h=0$ in $H^2$ is a holomorphic identity in the ball. Since $h$ is not identically zero, it is nonzero on a nonempty open set, on which the linear polynomial must vanish. Thus $a_1=a_2=0$. This contradiction uses $n\geq1$.

It follows that all $r_\beta$ vanish, and therefore $\Tcal(h)=0$. Equality gives $\Scal(h)=0$, and \eqref{eq:S-T} gives $\Rcal(h)=0$. All positive-degree homogeneous components of $h$ vanish, so $h$ is constant. Conversely, a constant of norm one satisfies the hypotheses and gives equality.

It remains to prove that the coefficient $2n+2$ in \eqref{eq:moment-estimate} is optimal. We use the following two reproducing kernels:
\[
k_{r,+}(z)=(1-rz_1)^{-(n+1)},\qquad
k_{r,-}(z)=(1+rz_1)^{-(n+1)},\qquad 0<r<1.
\]
The kernel formula \eqref{eq:reproducing-kernel}, summed over all degrees, gives
\[
\norm{k_{r,+}}_2^2=\norm{k_{r,-}}_2^2=A_r:=(1-r^2)^{-(n+1)},
\qquad \ip{k_{r,+}}{k_{r,-}}=B_r:=(1+r^2)^{-(n+1)}.
\]
Define
\[
u_r=\frac{k_{r,+}+k_{r,-}}{\sqrt{2(A_r+B_r)}},\qquad
v_r=\frac{k_{r,+}-k_{r,-}}{\sqrt{2(A_r+B_r)}},\qquad
q_r=\frac{A_r-B_r}{A_r+B_r}.
\]
Then $\norm{u_r}_2=1$, $\ip{u_r}{v_r}=0$, and $\norm{v_r}_2^2=q_r$. Both functions are smooth on the closed ball. For the kernel $k_a(z)=(1-z\cdot\bar a)^{-(n+1)}$, our inner-product convention gives $\ip{k_a}g=g(a)$. Thus for every $g\in H^2$,
\[
\ip{S_j^*k_a}g=\ip{k_a}{S_jg}=a_jg(a)
=\ip{\overline{a_j}k_a}g.
\]
Consequently $S_j^*k_a=\overline{a_j}k_a$. For the two real points $a=\pm re_1$, this becomes
\[
S_1^*k_{r,\pm}=\pm r k_{r,\pm},\qquad
S_j^*k_{r,\pm}=0\quad(j\geq2).
\]
$(S_1^*)^ku_r=r^ku_r$ for even $k$ and $(S_1^*)^ku_r=r^kv_r$ for odd $k$; $S_j^*u_r=0$ for $2\leq j\leq n+1$. Thus $u_r$ has zero first moments, and its only nonzero positive-degree moments are
$m_{2m\varepsilon_1}=r^{2m}$, and
\[
\Scal(u_r)=\frac{r^4}{1-r^4},\qquad
\frac1{n+1}\Rcal(u_r)=\frac{r^4+q_rr^2}{1-r^4},
\]
where the second equality follows from the summation identity \eqref{eq:shift-sum} on the Hardy space $H^2(\Ball)$. Its hypothesis $\Rcal(u_r)<\infty$ holds because $u_r$ is smooth on the closed ball; alternatively its homogeneous coefficients decay geometrically for each fixed $r<1$. Since $q_r\to1$ as $r\uparrow1$,
\begin{equation}\label{eq:moment-sharpness}
\frac{\Rcal(u_r)}{\Scal(u_r)}
=(n+1)\left(1+\frac{q_r}{r^2}\right)\longrightarrow2(n+1).
\end{equation}
Hence the coefficient cannot be increased. Theorem~\ref{thm:moment-estimate} applies to this family without requiring $u_r$ to be nonvanishing inside the ball. For comparison, if $n=0$, then $q_r=r^2$ and the ratio in \eqref{eq:moment-sharpness} equals $2$ for every $r$. The same nonconstant functions then give equality, explaining why the equality argument above requires $n\geq1$.
\end{proof}

Finally, combining the moment identity in Proposition~\ref{prop:Euler-moment} with Theorem~\ref{thm:moment-estimate}, we classify the extremals.

\begin{proposition}\label{prop:extremal-classification}
Let $n\geq1$ and $1/2\leq\alpha\leq1$. Every minimizer of $\Jf$ on $\Lzero$ is constant.
\end{proposition}
\begin{proof}
Let $F$ be a minimizer of $\Jf$ on $\Lzero$. Then $F$ is a constrained critical point in $\Dn$. By adding a constant, we may assume $Z(F)=1$. Propositions~\ref{prop:Euler} and~\ref{prop:eliminate} give smoothness and vanishing of the moment multiplier, respectively. The function $h=e^f$ in \eqref{eq:holomorphic-reduction} satisfies the hypotheses of Theorem~\ref{thm:moment-estimate}. Proposition~\ref{prop:Euler-moment} gives
\[
\frac{n+1}\alpha\Scal(h)=\Rcal(h)\geq2(n+1)\Scal(h).
\]
If $\alpha>1/2$ and $\Scal(h)>0$, division by $(n+1)\Scal(h)$ would imply $1/\alpha\geq2$, a contradiction. Thus $\Scal(h)=0$ and $\Rcal(h)=0$ in this case. If $\alpha=1/2$, the displayed identity is equality in Theorem~\ref{thm:moment-estimate}. In either case $h$ is constant. Since $h=e^f$ is nonzero, $F=\log\abs h^2$ is constant as well.
\end{proof}

The preceding proof uses the constrained criticality of $F$ and therefore also establishes Corollary~\ref{cor:critical-rigidity}.

\section{Proof of the Moser--Onofri inequality under zero-moment constraints}
In this section, we complete the proof of Theorem~\ref{thm:main}. We first use the existence and classification of extremals of the subcritical Moser--Onofri inequality under zero-moment constraints to establish $\Jf(F)\geq0$ for $\alpha>1/2$. Passing to the limit $\alpha\downarrow1/2$ then gives the sharp Moser--Onofri inequality under zero-moment constraints. We also characterize equality and prove that the coefficient is sharp by constructing a family whose normalized exponential densities concentrate at two antipodal points.

We begin with $1/2<\alpha\leq1$. Proposition~\ref{prop:existence} gives a minimizer of $\Jf$ on $\Lzero$, and Proposition~\ref{prop:extremal-classification} shows that it is constant. The value of $\Jf$ at a constant is zero. Therefore
\[
\Jf(F)\geq0\qquad(F\in\Lzero,\ 1/2<\alpha\leq1).
\]
If equality holds, $F$ is itself a global minimizer and hence is constant by Proposition~\ref{prop:extremal-classification}. When $\alpha>1$, the identity
\[
\Jf(F)=\Jone(F)+\frac{\alpha-1}{2(n+1)!}\E(F)
\]
proves nonnegativity. Equality forces $\E(F)=0$, and the spectrum \eqref{eq:spectrum} then implies that $F$ is constant.

We next consider the endpoint $\alpha=1/2$. Fix an arbitrary $F\in\Lzero$. Since $\E(F)<\infty$, the function $\alpha\mapsto\Jf(F)$ is affine and continuous. Letting $\alpha\downarrow1/2$ in the inequalities already proved gives
\[
\Jhalf(F)\geq0\qquad(F\in\Lzero).
\]
The limit is taken with the admissible function fixed.

We now determine the equality cases in Theorem~\ref{thm:main}. Suppose that $F\in\Lzero$ satisfies $\Jhalf(F)=0$. The inequality just proved shows that $F$ is a global minimizer of $\Jhalf$ on $\Lzero$, and hence a constrained critical point in $\Dn$. Since adding a constant preserves both the functional and the zero-moment condition, we may normalize $Z(F)=1$. Propositions~\ref{prop:Euler} and~\ref{prop:eliminate}, applied with $\alpha=1/2$, show that $F$ is smooth and satisfies \eqref{eq:Euler-free}.

Write $F=f+\bar f$ and $h=e^f$ as in \eqref{eq:holomorphic-reduction}. Then $h$ is holomorphic and smooth on the closed ball, $\norm h_2=1$, and $\int_{\Sph}\zeta\abs h^2\dd=0$. In particular, $\Rcal(h)<\infty$, so Theorem~\ref{thm:moment-estimate} applies. Proposition~\ref{prop:Euler-moment}, with $\alpha=1/2$, gives
\[
\Rcal(h)=2(n+1)\Scal(h).
\]
Thus equality holds in \eqref{eq:moment-estimate}. Since $n\geq1$, the equality statement in Theorem~\ref{thm:moment-estimate} forces $h$ to be constant. The normalization $\norm h_2=1$ then gives $e^F=\abs h^2=1$, hence $F=0$. Undoing the normalization shows that the original function $F$ is constant.

Conversely, every constant function $F\equiv c$ belongs to $\Lzero$ and satisfies
\[
\E(c)=0,\qquad Z(c)=e^c,\qquad
\Jhalf(c)=c-\log(e^c)=0.
\]
Therefore equality in Theorem~\ref{thm:main} holds if and only if $F$ is constant.

It remains to show that the coefficient $1/(4(n+1)!)$ in Theorem~\ref{thm:main} cannot be decreased. For $0<r<1$, define
\begin{equation}\label{eq:two-pole}
F_r(\zeta)=-2(n+1)\log\abs{1-r^2\zeta_1^2}.
\end{equation}
The factor $1-r^2z_1^2$ has no zero in a neighborhood of the closed ball. The logarithm with value zero at the origin gives the holomorphic function
\begin{equation}\label{eq:two-pole-series}
f_r(z)=-(n+1)\log(1-r^2z_1^2)
=(n+1)\sum_{m\geq1}\frac{r^{2m}}m z_1^{2m},\qquad F_r=f_r+\bar f_r.
\end{equation}
Thus $F_r$ is smooth and real pluriharmonic, with $\mean{F_r}=0$. The last assertion follows because every nonconstant holomorphic monomial has zero mean. Also $F_r(-\zeta)=F_r(\zeta)$, so the change of variables $\zeta\mapsto-\zeta$ gives
\[
\int_{\Sph}\zeta e^{F_r(\zeta)}\dd=-\int_{\Sph}\zeta e^{F_r(\zeta)}\dd=0.
\]
Hence $F_r\in\Lzero$ for every $0<r<1$.

To evaluate $\Jf(F_r)$, we first compute its energy. Write $L_r=-\log(1-r^4)$. The spectral and monomial formulas give, for $k\geq1$,
\[
\lambda_k\norm{z_1^k}_2^2
=\frac{(k+n)!}{(k-1)!}\frac{n!k!}{(k+n)!}
=kn!.
\]
The holomorphic and antiholomorphic components of every positive degree are orthogonal. Substitution of \eqref{eq:two-pole-series} into \eqref{eq:form-domain} gives
\begin{align}
\frac1{2(n+1)!}\E(F_r)
&=\frac1{(n+1)!}\sum_{m\geq1}\lambda_{2m}
\frac{(n+1)^2r^{4m}}{m^2}\norm{z_1^{2m}}_2^2\notag\\
&=\frac1{(n+1)!}\sum_{m\geq1}\frac{(n+1)^2r^{4m}}{m^2}\,2mn!
=2(n+1)\sum_{m\geq1}\frac{r^{4m}}m=2(n+1)L_r.
\label{eq:two-pole-energy}
\end{align}
The series are absolutely convergent for each $r<1$.

We next determine the asymptotic behavior of $Z(F_r)$ as $r\uparrow1$. The standard CR Jacobian formula, with normalized surface measure, gives
\[
I_r:=\int_{\Sph}|1-r\zeta_1|^{-2(n+1)}\dd=(1-r^2)^{-(n+1)}.
\]
Indeed, $(1-r^2)^{n+1}|1-r\zeta_1|^{-2(n+1)}$ is the Jacobian of a boundary CR automorphism and has integral one; see \cite{BFM}. We will prove the more precise two-pole asymptotic
\begin{equation}\label{eq:two-pole-mass-asymptotic}
Z(F_r)=2^{-n}(1-r^4)^{-(n+1)}\bigl(1+o(1)\bigr)
\qquad(r\uparrow1).
\end{equation}

Let $p_+=(1,0,\ldots,0)$, $H_+=\{\Rea\zeta_1\geq0\}$, and $H_-=-H_+$. Define
\[
p_r(\zeta)=I_r^{-1}|1-r\zeta_1|^{-2(n+1)},\qquad
 g_r(\zeta)=|1+r\zeta_1|^{-2(n+1)}\quad(\zeta\in H_+).
\]
The function $p_r$ is nonnegative and has integral one. For every fixed sufficiently small open neighborhood $U$ of $p_+$, the continuous function $|1-\zeta_1|$ has a strictly positive minimum on the compact set $\Sph\setminus U$, since $\zeta_1=1$ on the sphere only at $p_+$. Denote this minimum by $\delta_U>0$. If $1-r\leq\delta_U/2$, then
$|1-r\zeta_1|\geq|1-\zeta_1|-(1-r)\geq\delta_U/2$ on $\Sph\setminus U$. Consequently,
\begin{equation}\label{eq:pole-concentration}
\int_{\Sph\setminus U}p_r\dd
\leq\left(\frac2{\delta_U}\right)^{2(n+1)}(1-r^2)^{n+1}\longrightarrow0.
\end{equation}
The number $\delta_U$ is fixed before $r\uparrow1$.

On $H_+$ we have $|1+r\zeta_1|\geq1$. Hence $0<g_r\leq1$, and $g_r$ converges uniformly on $H_+$ to $g_1(\zeta)=|1+\zeta_1|^{-2(n+1)}$. In particular $g_1(p_+)=2^{-2(n+1)}$. We next justify the integral limit using this concentration. Given $\epsilon>0$, choose an open neighborhood $U\subset\{\Rea\zeta_1>0\}$ of $p_+$ so that
$|g_1(\zeta)-2^{-2(n+1)}|\leq\epsilon$ on $U$. For $r$ sufficiently close to $1$, uniform convergence gives $|g_r-2^{-2(n+1)}|\leq2\epsilon$ there. Since $\int p_r\dd=1$,
\[
\left|\int_{H_+}g_rp_r\dd-2^{-2(n+1)}\right|
\leq2\epsilon+3\int_{\Sph\setminus U}p_r\dd.
\]
Letting $r\uparrow1$ and then $\epsilon\downarrow0$, and using \eqref{eq:pole-concentration}, yields
\[
\frac1{I_r}\int_{H_+}e^{F_r}\dd
=\int_{H_+}g_rp_r\dd\longrightarrow2^{-2(n+1)}.
\]
The change of variables $\zeta\mapsto-\zeta$ shows that the two hemisphere integrals of $e^{F_r}$ are equal. Their common boundary has surface measure zero, so
\[
\frac{Z(F_r)}{I_r}\longrightarrow2^{-2n-1}.
\]
Finally $I_r=(1+r^2)^{n+1}(1-r^4)^{-(n+1)}$ and $(1+r^2)^{n+1}\to2^{n+1}$. This gives \eqref{eq:two-pole-mass-asymptotic} and shows that the leading contribution comes from neighborhoods of the two antipodal poles. Taking logarithms gives
\begin{equation}\label{eq:two-pole-logmass}
\log Z(F_r)=(n+1)L_r-n\log2+o(1).
\end{equation}

Combining \eqref{eq:two-pole-energy} and \eqref{eq:two-pole-logmass}, and using $\mean{F_r}=0$, we obtain
\begin{equation}\label{eq:sharpness}
\Jf(F_r)=(2\alpha-1)(n+1)L_r+ n\log2+o(1).
\end{equation}
As $L_r\to\infty$, every $0<\alpha<1/2$ gives $\Jf(F_r)\to-\infty$. This shows that the endpoint value $\alpha=1/2$ is optimal in Theorem~\ref{thm:main}.

At the endpoint the same computation gives
\[
\Jhalf(F_r)\longrightarrow n\log2,
\qquad\E(F_r)\longrightarrow\infty.
\]
For $n\geq1$, the endpoint functional $\Jhalf(F_r)$ has a positive limit, so it is not a minimizing sequence for the endpoint problem.

\section{Proof of the stability of the Moser--Onofri inequality on the CR sphere}\label{sec:global-stability}
We deduce Theorem~\ref{thm:stability} directly from Theorem~\ref{thm:main}
by the conformal argument of Carlen
\cite{Carlen25}. We keep the original value of $Z(F)$
throughout. The two identities needed are invariance of $\Jone$ under the
affine action and invariance of $\E$ under composition, as recorded in
\eqref{eq:affine-action} and \eqref{eq:energy-composition}.

\begin{proof}[Proof of Theorem~\ref{thm:stability}]
Fix any real $F\in\Dn$. The known exponential estimate gives
$0<Z(F)<\infty$. Apply \eqref{eq:CR-centering-density} to the probability
density $p=e^F/Z(F)$. For the resulting $\tau\in\Aut(\Sph)$, define
\begin{equation}\label{eq:stability-centered-transform}
F_\tau=F\circ\tau+\log J_\tau\in\Dn.
\end{equation}
Multiplying \eqref{eq:CR-centering-density} by $Z(F)$ and using the
Jacobian change of variables gives
\begin{equation}\label{eq:stability-center-and-mass}
\int_{\Sph}\zeta e^{F_\tau}\dd=0,
\qquad Z(F_\tau)=Z(F).
\end{equation}
In particular $F_\tau\in\Lzero$. Theorem~\ref{thm:main} applies without
requiring its exponential mass to be one. Therefore
\begin{equation}\label{eq:stability-centered-bound}
\Jone(F_\tau)
=\Jhalf(F_\tau)+\frac{1}{4(n+1)!}\E(F_\tau)
\geq\frac{1}{4(n+1)!}\E(F_\tau).
\end{equation}

The inverse-Jacobian formula in Subsection~\ref{sec:automorphisms} says
\[
(\log J_{\tau^{-1}})\circ\tau=-\log J_\tau.
\]
Consequently
\begin{equation}\label{eq:stability-difference}
F_\tau=(F-\log J_{\tau^{-1}})\circ\tau,
\qquad
\E(F_\tau)=\E(F-\log J_{\tau^{-1}}),
\end{equation}
where the energy identity is \eqref{eq:energy-composition}. Since
$\log J_{\tau^{-1}}\in\mathcal M$, conformal invariance of $\Jone$ and
\eqref{eq:stability-centered-bound} imply
\begin{align}
\Jone(F)
&=\Jone(F_\tau)\notag\\
&\geq\frac{1}{4(n+1)!}\E(F-\log J_{\tau^{-1}})\notag\\
&\geq\frac{1}{4(n+1)!}\inf_{h\in\mathcal M}\E(F-h).
\label{eq:global-stability-chain}
\end{align}
This is \eqref{eq:stability}. It does not require the infimum over
$\mathcal M$ to be attained.

Suppose that equality holds in \eqref{eq:stability}. The first and last
terms of \eqref{eq:global-stability-chain} are then equal, so equality
must also hold in \eqref{eq:stability-centered-bound}. Thus
$\Jhalf(F_\tau)=0$. The equality statement of Theorem~\ref{thm:main}
gives $F_\tau\equiv c$ for some $c\in\R$. Applying the inverse action,
\[
F=c+\log J_{\tau^{-1}}\in\mathcal M.
\]
In fact $c=\log Z(F)$ by \eqref{eq:stability-center-and-mass}. Conversely,
if $F\in\mathcal M$, then $\Jone(F)=0$ by Theorem~C, while choosing
$h=F$ makes the infimum in \eqref{eq:stability} zero. Hence every member
of $\mathcal M$, not only the constants, gives the equality.
\end{proof}
{\bf Acknowledgement.} The first was partly supported by National Natural Science Foundation of China (No. 12622108) the second author was partly supported by a grant from Simons Foundation.


\begin{thebibliography}{99}
\setlength{\itemsep}{0pt}
\bibitem{Aubin} T. Aubin, \emph{Meilleures constantes dans le th\'eor\`eme d'inclusion de Sobolev et un th\'eor\`eme de Fredholm non lin\'eaire pour la transformation conforme de la courbure scalaire}, J. Funct. Anal. \textbf{32} (1979), 148--174.
\bibitem{Beckner} W. Beckner, \emph{Sharp Sobolev inequalities on the sphere and the Moser--Trudinger inequality}, Ann. of Math. (2) \textbf{138} (1993), 213--242.
\bibitem{BFM} T. P. Branson, L. Fontana, and C. Morpurgo, \emph{Moser--Trudinger and Beckner--Onofri's inequalities on the CR sphere}, Ann. of Math. (2) \textbf{177} (2013), 1--52.
\bibitem{Carlen17} E. A. Carlen, \emph{Duality and stability for functional inequalities}, Ann. Fac. Sci. Toulouse Math. (6) \textbf{26} (2017), 319--350.
\bibitem{Carlen25} E. A. Carlen, \emph{Stability for the logarithmic Hardy--Littlewood--Sobolev inequality with application to the Keller--Segel equation}, J. Funct. Anal. \textbf{288} (2025), Paper No. 110818.
\bibitem{CY87} S.-Y. A. Chang and P. C. Yang, \emph{Prescribing Gaussian curvature on $S^2$}, Acta Math. \textbf{159} (1987), 215--259.
\bibitem{CY95} S.-Y. A. Chang and P. C. Yang, \emph{Extremal metrics of zeta function determinants on 4-manifolds}, Ann. of Math. (2) \textbf{142} (1995), 171--212.
\bibitem{ChenLuTang} L. Chen, G. Lu, and H. Tang, \emph{Sharp stability of log-Sobolev and Moser--Onofri inequalities on the sphere}, J. Funct. Anal. \textbf{285} (2023), Paper No. 110022.
\bibitem{ChenLuTangWang2026} L. Chen, G. Lu, H. Tang, and B. Wang, \emph{Asymptotically sharp stability of Sobolev inequalities on the Heisenberg group with dimension-dependent constants}, J. Math. Pures Appl. \textbf{206} (2026), 103832.
\bibitem{CL} W. S. Cohn and G. Lu, \emph{Sharp constants for Moser--Trudinger inequalities on spheres in complex space $\C^n$}, Comm. Pure Appl. Math. \textbf{57} (2004), 1458--1493.
\bibitem{Folland} G. B. Folland, \emph{Spherical harmonic expansion of the Poisson--Szeg\H{o} kernel for the ball}, Proc. Amer. Math. Soc. \textbf{47} (1975), 401--408.
\bibitem{FS} G. B. Folland and E. M. Stein, \emph{Estimates for the $\bar\partial_b$ complex and analysis on the Heisenberg group}, Comm. Pure Appl. Math. \textbf{27} (1974), 429--522.
\bibitem{FL} R. L. Frank and E. H. Lieb, \emph{Sharp constants in several inequalities on the Heisenberg group}, Ann. of Math. (2) \textbf{176} (2012), 349--381.
\bibitem{GLWY} C. Gui, T. Li, J. Wei, and Z. Ye, \emph{A positive answer to the generalized Chang--Yang conjecture on $S^N$}, preprint, arXiv:2608.13497 (2026).
\bibitem{GM} C. Gui and A. Moradifam, \emph{The sphere covering inequality and its applications}, Invent. Math. \textbf{214} (2018), 1169--1204.
\bibitem{JL87} D. Jerison and J. M. Lee, \emph{The Yamabe problem on CR manifolds}, J. Differential Geom. \textbf{25} (1987), 167--197.
\bibitem{JL88} D. Jerison and J. M. Lee, \emph{Extremals for the Sobolev inequality on the Heisenberg group and the CR Yamabe problem}, J. Amer. Math. Soc. \textbf{1} (1988), 1--13.
\bibitem{Onofri} E. Onofri, \emph{On the positivity of the effective action in a theory of random surfaces}, Comm. Math. Phys. \textbf{86} (1982), 321--326.
\bibitem{Rudin} W. Rudin, \emph{Function Theory in the Unit Ball of $\C^n$}, Grundlehren Math. Wiss., vol. 241, Springer, New York, 1980.
\bibitem{TaylorI} M. E. Taylor, \emph{Partial Differential Equations I: Basic Theory}, 2nd ed., Applied Mathematical Sciences, vol. 115, Springer, New York, 2011.
\bibitem{TaylorIII} M. E. Taylor, \emph{Partial Differential Equations III: Nonlinear Equations}, 2nd ed., Applied Mathematical Sciences, vol. 117, Springer, New York, 2011.
\bibitem{Zhu} K. Zhu, \emph{Spaces of Holomorphic Functions in the Unit Ball}, Graduate Texts in Mathematics, vol. 226, Springer, New York, 2005.
\end{thebibliography}
\end{document}